\documentclass[12pt]{amsart} 
\usepackage{amsmath, amssymb, mathscinet,mathtools}
\usepackage{amsthm} 
\usepackage[all]{xy}
\usepackage{enumerate}
\usepackage[T1]{fontenc}

\usepackage{hyperref} 
\usepackage{mathrsfs} 

\usepackage{graphicx}
\usepackage{color}
\usepackage{comment}
\theoremstyle{plain} 
\newtheorem{theorem}{\indent\sc Theorem}[section]
\newtheorem{lemma}[theorem]{\indent\sc Lemma}
\newtheorem{corollary}[theorem]{\indent\sc Corollary}
\newtheorem{proposition}[theorem]{\indent\sc Proposition}

\theoremstyle{definition} 
\newtheorem{definition}[theorem]{\indent\sc Definition}
\newtheorem{remark}[theorem]{\indent\sc Remark}
\newtheorem{example}[theorem]{\indent\sc Example}
\newtheorem{notation}[theorem]{\indent\sc Notation}

\makeatletter
\def\address#1#2{\begingroup
\noindent\parbox[t]{7.8cm}{%
\small{\scshape\ignorespaces#1}\par\vskip1ex
\noindent\small{\itshape E-mail address}%
\/: #2\par\vskip4ex}\hfill%
\endgroup}%
\makeatother
\title[]{\uppercase{Coarse Geometry of Free Products of Metric Spaces}}
\author{
\textsc{Qin Wang and Jvbin Yao} 
}
\date{} 
\begin{document}


\begin{abstract}


Recently, a notion of the free product $X \ast Y$ of two metric spaces $X$ and $Y$ has been introduced by T. Fukaya and T. Matsuka. In this paper, we study coarse geometric permanence properties of the free product $X \ast Y$. We show that if $X$ and $Y$ satisfy any of the following conditions, then $X \ast Y$ also satisfies that condition: (1) they are coarsely embeddable into a Hilbert space or a uniformly convex Banach space; (2) they have Yu's Property A; (3) they are hyperbolic spaces. These generalize the corresponding results for discrete groups to the case of metric spaces.

\end{abstract}

\maketitle

\section{Introduction}
\label{sec:introduction}

  The study of metric spaces and their embeddings into various spaces, particularly Hilbert spaces and uniformly convex Banach spaces, has significant implications in several areas of mathematics. This paper focuses on the behavior of certain properties of metric spaces under the operation of taking the free product. Specifically, we investigate the preservation of coarse embeddability into Hilbert spaces or, more generally, uniformly convex Banach spaces, Yu's Property A, and the hyperbolicity of free products of hyperbolic spaces.

  The concept of coarse embeddability of a metric space $(X, d_X)$ into a Hilbert space, or more generally into a uniformly convex Banach space, was introduced by Gromov \cite{gromov1992asymptotic}. Let us recall the definition in the latter setting: a metric space $(X,d_{X})$ is coarsely embeddable into a uniformly convex Banach space $E$ if there exist a map $F:X\longrightarrow E$ and non-decreasing functions $\rho_{1}, \rho_{2}: \mathbb{R}_{+}\longrightarrow \mathbb{R}$ satisfying $\lim_{t \to+\infty}\ \rho_{i}(t)=+\infty$ (for $i=1,2$) such that for all $x,y\in X$,
  \[
  \rho_{1}(d_{X}(x,y))\leq \Vert F(x)-F(y)\Vert_E\leq \rho_{2}(d_{X}(x,y)).
  \]
   Significant motivation for studying such embeddings stems from their connections to important conjectures in topology and geometry. For instance, Yu \cite{yu2000coarse} showed that a discrete metric space with bounded geometry admitting a coarse embedding into a Hilbert space satisfies the coarse Baum-Connes conjecture. Subsequently, Kasparov and Yu established that the coarse geometric Novikov conjecture holds for discrete metric spaces with bounded geometry admitting a coarse embedding into a uniformly convex Banach space \cite{kasparov2006coarse}. The importance of coarse embeddings has led to extensive research on this topic. An important question in the study of coarse embeddings concerns their preservation under constructions like the free product. For countable discrete groups, it is known that if $A$ and $B$ are coarsely embeddable into Hilbert space, so is their free product $\Gamma = A \ast B$ \cite{chen2003uniform}. Similarly, if $A$ and $B$ are uniformly embeddable into a uniformly convex Banach space, so is their free product $\Gamma$ \cite{chen2006uniform}. 
  
    We further discuss two important concepts, Property A and hyperbolicity. Property A, introduced by G. Yu in \cite{yu2000coarse} as a weak form of amenability, serves as a sufficient condition for coarse embeddability into Hilbert space. Yu showed that a metric space $X$ with Property A coarsely embeds into Hilbert space, and that this coarse embeddability implies the coarse Baum-Connes conjecture, and consequently the strong Novikov conjecture if $X$ is the underlying coarse space of a countable discrete group $G$ (\cite{higson2000bivariant}, \cite{skandalis2002coarse}). Similar to coarse embeddability, Property A for countable discrete groups is also preserved under free products. A hyperbolic metric space (whose rigorous definition will be given in Section 5), introduced by Mikhael Gromov in \cite{gromov1987hyperbolic}, is a metric space in which points satisfy certain metric inequalities, quantitatively controlled by a non-negative real number $\delta$. This definition generalizes the metric properties of classical hyperbolic geometry and trees, providing a large-scale property that is essential for the analysis of certain infinite groups, known as Gromov hyperbolic groups. It was shown in \cite{gromov1987hyperbolic} that the free product of two hyperbolic groups is also hyperbolic.
    
      Recently, a notion of the free product of metric spaces was formally introduced by Tomohiro Fukaya and Takumi Matsuka in \cite{fukaya2023free}, who proved that the free product of symmetric geodesic coarsely convex spaces preserves this property and further showed that it satisfies the coarse Baum-Connes conjecture.  Given the significance of the free product construction and these initial findings by T. Fukaya and T. Matsuka, it is natural to consider how key properties such as coarse embeddability, Property A, and hyperbolicity behave under the free product of metric spaces. Therefore, this paper explores the following questions:

    \begin{enumerate}
	\item \textbf{Coarse Embeddability of Free Products:} If metric spaces $(X, d_X)$ and $(Y, d_Y)$ are both coarsely embeddable in a uniformly convex Banach space, is their free product $X\ast Y$ also coarsely embeddable? We provide a positive answer. This generalizes the result for countable discrete groups in \cite{chen2006uniform}.
	
	\item \textbf{Property A for Free Products:}  If $X$ and $Y$ both possess Property A, does $X \ast Y$ also possess Property A? We show that this is indeed the case, extending the known results for groups \cite{chen2003uniform}.
	
	\item \textbf{Hyperbolicity of Free Products:} Finally, we consider the case where $X$ and $Y$ are hyperbolic spaces. Hyperbolicity, a large-scale geometric property \cite{nowak2023large}, is preserved under certain constructions. We prove that the free product of two hyperbolic spaces remains hyperbolic.
   \end{enumerate}

   We summarize the main results of this paper in the following theorem:
\begin{theorem}
Let $(X,d_{X})$ and $(Y,d_{Y})$ be metric spaces with nets $(X_{0},i_{X},e_{X})$ and $(Y_{0},i_{Y},e_{Y})$ of $X$ and $Y$, respectively.
\begin{itemize}
	\item If $X$ and $Y$ are coarsely embeddable into a uniformly convex Banach space, then so is their free product $X\ast Y$.
	\item If $X$ and $Y$ have Property A, then $X\ast Y$ has Property A.
	\item If $X$ and $Y$ have Gromov hyperbolicity, then $X\ast Y$ has Gromov hyperbolicity.
\end{itemize}
\end{theorem}

This paper is organized as follows. In Section 2, we recall the concept of the free product of metric spaces, which introduced by T. Fukaya and T. Matsuka. We provide a detailed description of the construction of this free product and clarify the relationship between the free product of metric spaces and the free product of groups when the groups are equipped with a word length metric. It should be noted that all aspects of the free product construction presented in this section are based on the work of T. Fukaya and T. Matsuka\cite{fukaya2023free}. In Section 3, we present a direct proof of the first main result (Coarse Embeddability of Free Products), using the properties of coarse embeddings. In Section 4, by viewing the free product of metric spaces as  ``trees'' and by applying a lemma concerning trees and Hilbert spaces, along with characterizations of coarse embeddability and Property A from \cite{willett2006some}, we provide a unified proof for both the first and second main results (Coarse Embeddability and Property A of Free Products). Finally, in Section 5, we first present a concise proof demonstrating that the free product of geodesic hyperbolic spaces remains a geodesic hyperbolic space. We then extend this result to general hyperbolic spaces.

 Our work contributes to the understanding of how these important geometric and analytic properties behave under the free product construction. This contribution provides further insights into large-scale geometry.

  \section{free products of metrics spaces}

In this section, we will recall the definition of free products of metric spaces, which was first introduced in \cite{fukaya2023free}.
\begin{definition}[Net of a Metric Space]
	Let $(X,d_{X})$ be a metric space. A net of $X$ is a triple $(X_{0},i_{X},e_{X})$ satisfying the following conditions:
	\begin{itemize}
		\item $X_{0}$ is a set.
		\item $i_{X}:X_{0} \rightarrow X$ is a map such that for any compact subset $K \subseteq X$, the preimage $i_{X}^{-1}(K)$ is a finite set. The map $i_{X}$ is referred to as the index map.
		\item $e_{X}\in i_{X}(X_{0})$. The element $e_{X}$ is referred to as a base point.
	\end{itemize}
	For any $x_{0}\in X_{0}$, we denote its image $i_{X}(x_{0})$ by $\overline{x_{0}}$.
\end{definition}

\begin{definition}[Normal Word]
	Let $(X,d_{X})$ and $(Y,d_{Y})$ be metric spaces with nets $(X_{0},i_{X},e_{X})$ and $(Y_{0},i_{Y},e_{Y})$ of $X$ and $Y$, respectively. We choose $\epsilon_X\in X_{0}$ and $\epsilon_Y\in Y_{0}$ such that $i_X(\epsilon_X) = e_X$ and $i_Y(\epsilon_Y) = e_Y$, respectively. Let $X_0^* = X_0 \setminus \{\epsilon_X\}$ and $Y_0^* = Y_0 \setminus \{\epsilon_Y\}$. A \textit{normal word} on $X_0^* \sqcup Y_0^*$ is a finite sequence
	\begin{align*}
		w_0 w_1 \cdots w_n \quad (n \geq 0, w_i \in X_0^* \sqcup Y_0^*)
	\end{align*}
	such that for all $i$, we have
	\begin{align*}
		w_i \in X_0^* &\Rightarrow w_{i+1} \in Y_0^*, \\
		w_i \in Y_0^* &\Rightarrow w_{i+1} \in X_0^*.
	\end{align*}
	
	Let $\omega = w_{0}w_{1}\dots w_{n}$ be a normal word. We define the length of $\omega$ to be $n+1$. Let $\epsilon$ denote the empty word.
\end{definition}

\begin{remark}
	The conditions defining normal words are analogous to those defining reduced words in free products of groups.
\end{remark}

\begin{definition}[Sheets in Free Products]
	Let $W$ be the set consisting of the empty word $\epsilon$ and all normal words. We define $W_{X}$ to be the subset of $W$ consisting of the empty word $\epsilon$ and normal words whose last letter does not belong to $X_{0}^*$, that is,
	\begin{align*}
		W_{X} &:= \{w = w_{0}w_{1}\dots w_{n}\in W \mid (w \neq \epsilon \implies w_{n} \notin X_{0}^*) \} \sqcup \{\epsilon\} \\
		&= \left\{w_{0}w_{1}\dots w_{n}\in W \mid w_{n}\notin X_{0}^*\right\} \sqcup \{\epsilon\}.
	\end{align*}
	Similarly, we define $W_{Y}$ as
	\begin{align*}
		W_{Y} &:= \{w = w_{0}w_{1}\dots w_{n}\in W \mid (w \neq \epsilon \implies w_{n} \notin Y_{0}^*) \} \sqcup \{\epsilon\} \\
		&= \left\{w_{0}w_{1}\dots w_{n}\in W \mid w_{n}\notin Y_{0}^*\right\} \sqcup \{\epsilon\}.
	\end{align*}
	Then we have
	\[
	W_{X}\times X = \bigsqcup_{\omega\in W_{X}} \{\omega\}\times X  ,\quad W_{Y}\times Y = \bigsqcup_{\tau\in W_{Y}} \{\tau\}\times Y.
	\]
	For any $\omega\in W_{X}$, the set $\{\omega\}\times X \subseteq W_{X}\times X$, denoted simply by $\omega X$, is called a \textit{sheet} of $X$. Similarly, a sheet of $Y$ is defined as a subset of $W_{Y}\times Y$. We identify $X$ and $Y$ with $\epsilon X$ and $\epsilon Y$, respectively. We refer to $\omega$ and $\tau$ as \textit{index words}. For each sheet, the \textit{height} of the sheet is defined to be the length of the index word.
	
\end{definition}

Having defined sheets in free products of metric spaces, we next define edges in these free products to connect these sheets. To manage the connections between sheets and edges, we will subsequently define an equivalence relation on them.

\begin{definition}[Edges and Equivalence Relation]
	We define the set of \textit{edges} to be the disjoint union
	\[
	(W_{X}\times X_{0}\times [0,1]) \sqcup (W_{Y}\times Y_{0}\times [0,1]) \sqcup (\epsilon\times [0,1]).
	\]
	We consider the disjoint union $\widetilde{X\ast  Y}$ of all sheets and edges, defined as
	\[
	\widetilde{X\ast  Y} := (W_{X}\times X) \bigsqcup (W_{Y}\times Y) \bigsqcup (W_{X}\times X_{0}^*\times [0,1]) \bigsqcup (W_{Y}\times Y_{0}^*\times [0,1]) \bigsqcup (\epsilon\times [0,1]).
	\]
	We define an equivalence relation $\sim$ on $\widetilde{X\ast  Y}$ by specifying the following generating relations:
	\begin{itemize}
		\item  $(\epsilon,e_{X}) \sim (\epsilon,0)$ and $(\epsilon,e_{Y}) \sim (\epsilon,1)$ (see Figure~\ref{figure1}).
		\begin{figure}[htbp]
			\centering
			\includegraphics[width=0.6\textwidth]{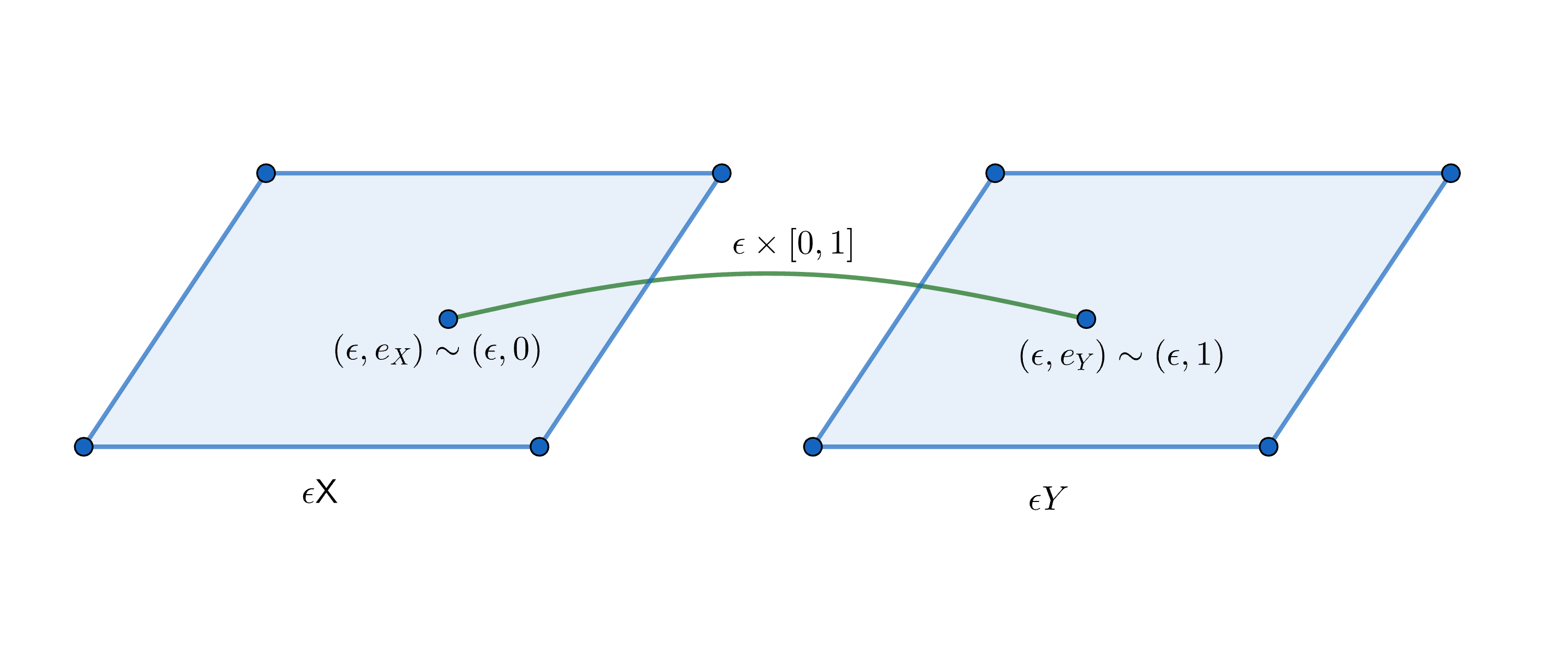}
			\caption{Since $(\epsilon, e_X) \sim (\epsilon, 0)$ and $(\epsilon, e_Y) \sim (\epsilon, 1)$, $(\epsilon, e_X)$ and $(\epsilon, 0)$ are depicted by the same point in the graph, and similarly for $(\epsilon, e_Y)$ and $(\epsilon, 1)$.}
			\label{figure1}
		\end{figure}
		\item For $\omega \in W_{X}$ and $x_{0}\in X_{0}^*$, we have the following equivalences:
	\begin{align*}
		(\omega,\overline{x_{0}}) &\sim (\omega,x_{0},0) \\
		&\text{where } (\omega,\overline{x_{0}}) \in \{\omega\}\times X \text{ and } (\omega,x_{0},0) \in \{\omega\}\times \{x_{0}\}\times [0,1], \\
		(\omega x_{0},e_{Y}) &\sim (\omega,x_{0},1) \\
		&\text{where } (\omega x_{0},e_{Y}) \in \{\omega x_{0}\}\times Y \text{ and } (\omega,x_{0},1) \in \{\omega\}\times \{x_{0}\}\times [0,1]
	\end{align*}
		(see Figure~\ref{figure2}).
		\begin{figure}[htbp]
			\centering
			\includegraphics[width=0.6\textwidth]{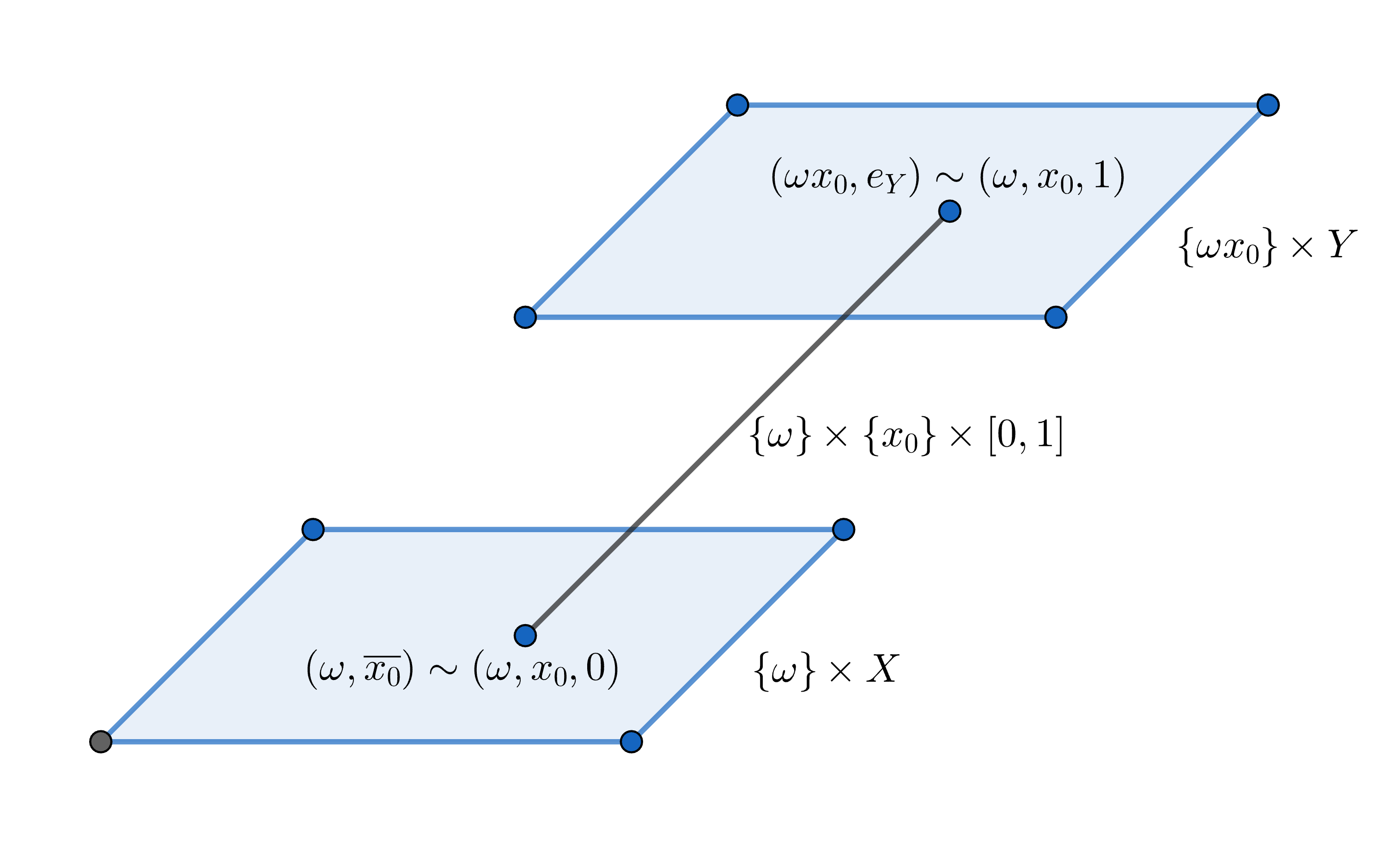}
			\caption{$(\omega,\overline{x_{0}})$ and $(\omega,x_{0},0)$ are depicted by the same point in the graph, and similarly for $(\omega x_{0},e_{Y})$ and $(\omega,x_{0},1)$}
			\label{figure2}
		\end{figure}
		\item For $\tau \in W_{Y}$ and $y_{0}\in Y_{0}^*$, we have the following equivalences:
		\begin{align*}
		(\tau,\overline{y_{0}}) &\sim (\tau,y_{0},0) \\
		&\text{where } (\tau,\overline{y_{0}}) \in \{\tau\}\times Y \text{ and } (\tau,y_{0},0) \in \{\tau\}\times \{y_{0}\}\times [0,1], \\
		(\tau y_{0},e_{X}) &\sim (\tau,y_{0},1) \\
		&\text{where } (\tau y_{0},e_{X}) \in \{\tau y_{0}\}\times X \text{ and } (\tau,y_{0},1) \in \{\tau\}\times \{y_{0}\}\times [0,1].
	\end{align*}
		(see Figure~\ref{figure3}).
		\begin{figure}[htbp]
			\centering
			\includegraphics[width=0.6\textwidth]{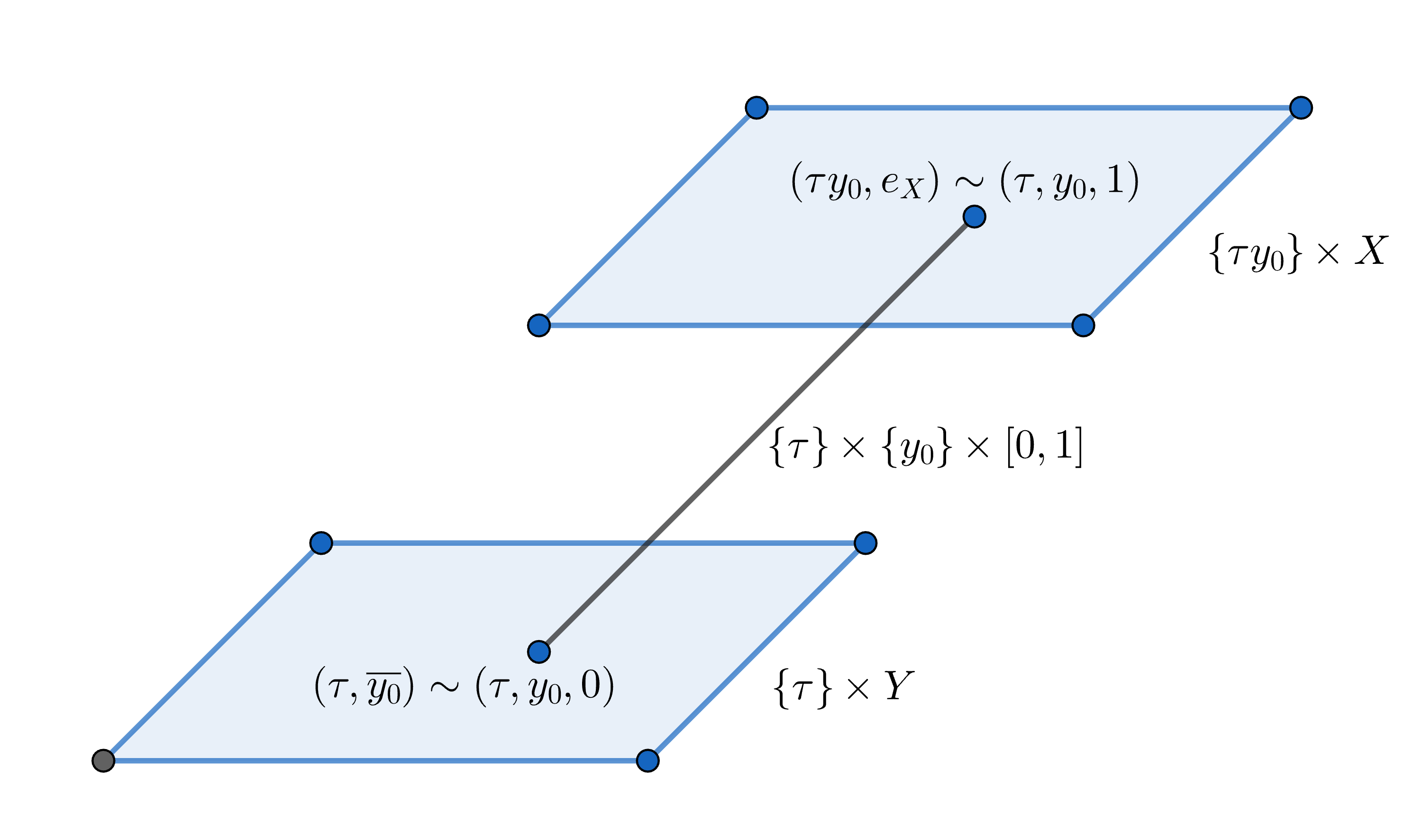}
			\caption{$(\tau,\overline{y_{0}})$ and $y_{0}\in Y_{0}^*$ are depicted by the same point in the graph, and similarly for $(\tau y_{0},e_{X})$ and $(\tau,y_{0},1)$}
			\label{figure3}
		\end{figure}
	\end{itemize}
\end{definition}

\begin{definition}[Free Product of Metric Spaces \cite{fukaya2023free}]
	We define the free product of metric spaces $X\ast Y$, to be the quotient space $\widetilde{X\ast  Y}/\sim$.
\end{definition}

\begin{proposition}[Components of the Free Product]
	The free product $X\ast Y$ is comprised of two types of components (see Figure \ref{figure4} for illustration):
	\begin{itemize}
		\item \textbf{Sheets:} These are sets of the form $\{\omega\}\times X$ and $\{\tau\}\times Y$, where $\omega \in W_{X}$ and $\tau \in W_{Y}$. For simplicity, we write $\{\omega\} \times X$ and $\{\tau\} \times Y$ as $\omega X$ and $\tau Y$, respectively.
		\item \textbf{Edges:} These are of the following three types:
		\begin{itemize}
			\item The edge $\{\epsilon\}\times [0,1]$ connecting $(\epsilon,e_{X})\in \{\epsilon\}\times X$ and $(\epsilon,e_{Y})\in \{\epsilon\}\times Y$.
			\item Let $\omega \in W_{X}$ and $x_{0} \in X_{0}^*$. Then there exists an edge $\{\omega\} \times \{x_{0}\} \times [0,1]$ connecting $(\omega,\overline{x_{0}}) \in \{\omega\} \times X$ and $(\omega x_{0}, e_{Y}) \in \{\omega x_{0}\} \times Y$.
			\item Let $\tau \in W_{Y}$ and $y_{0} \in Y_{0}^*$. Then there exists an edge $\{\tau\} \times \{y_{0}\} \times [0,1]$ connecting $(\tau,\overline{y_{0}}) \in \{\tau\} \times Y$ and $(\tau y_{0}, e_{X}) \in \{\tau y_{0}\} \times X$.
		\end{itemize}
	\end{itemize}
	\begin{figure}[htbp]
		\centering
		\includegraphics[width=0.8\textwidth]{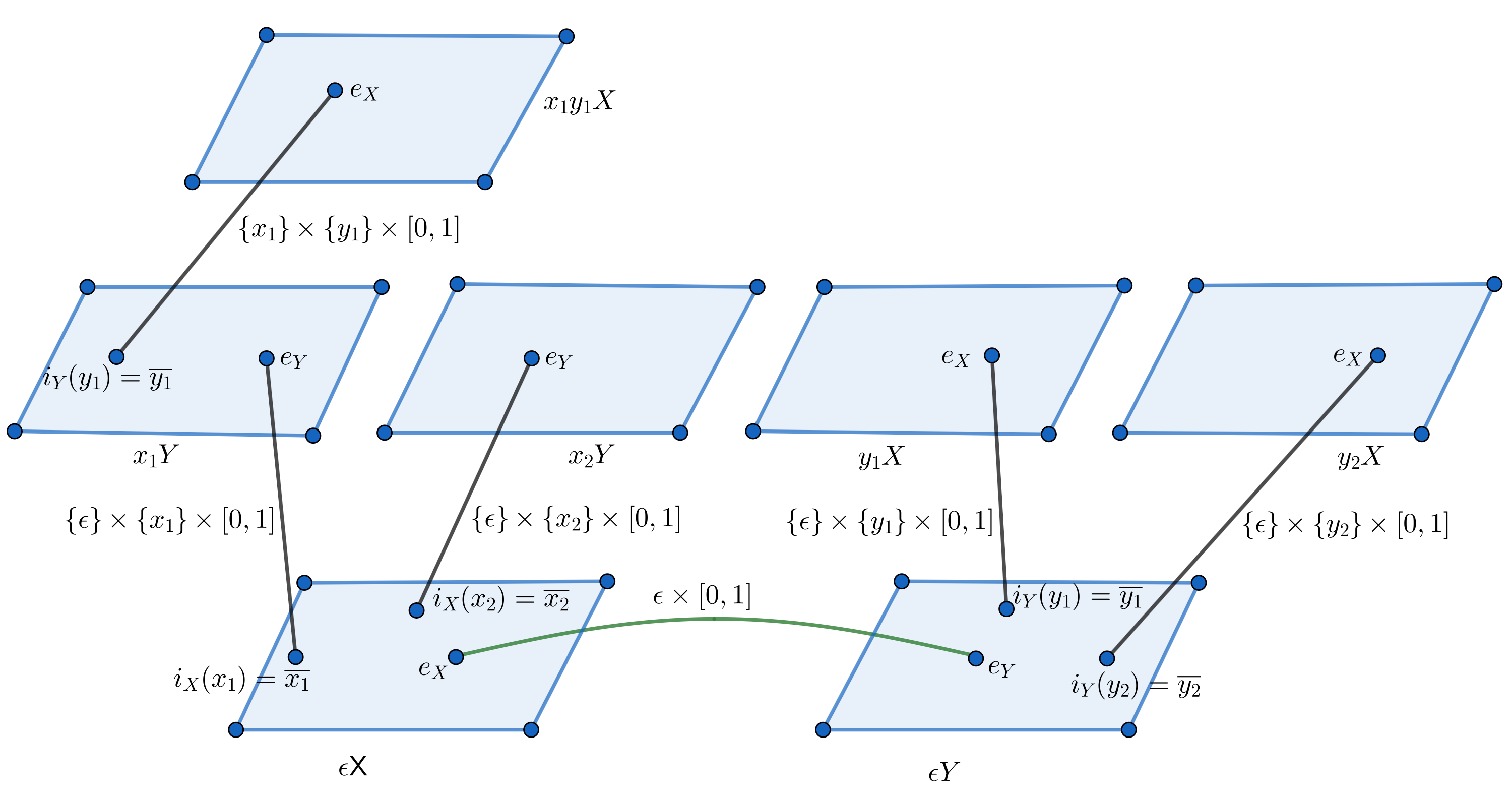}
		\caption{Components of the Free Product $X\ast Y$}
		\label{figure4}
	\end{figure}
	
\end{proposition}

\begin{notation}[Coordinate Representation]\label{coordinate}
	Each point $p\in X\ast Y$ can be canonically identified with a triplet $(\omega,z,t)$, where $\omega \in  W$, $z\in X\sqcup Y$, and $t\in [0,1)$, satisfying:
	\begin{itemize}
		\item If $\omega \in W_{X}$, then $z\in X$.
		\item If $\omega \in W_{Y}$, then $z\in Y$.
		\item If $z\notin i_{X}(X_{0})\sqcup i_{Y}(Y_{0})$, then $t=0$.
	\end{itemize}
	We refer to $(\omega,z,t)$ as the coordinate of $p$. A point $p$ is said to belong to a sheet if $t=0$, and to an edge if $t>0$. When $p$ belongs to a sheet, we abbreviate the coordinate $(\omega,z,0)$ as $(\omega,z)$.
\end{notation}

\begin{notation}[Distance Notations]
	For convenience, we introduce the following notations:
	\begin{itemize}
		\item For $z\in X\sqcup Y$, define
		\[
		\Vert z \Vert :=
		\begin{cases}
			d_{X}(e_{X},z),  & \text{if } z \in X \\
			d_{Y}(e_{Y},z), & \text{if } z \in Y
		\end{cases}
		\]
		\item For $u,v\in X\sqcup Y$, define
		\[
		d_{X\sqcup Y}(u,v) :=
		\begin{cases}
			d_{X}(u,v),& \text{if } \{u,v\}\subset X \\
			d_{Y}(u,v),& \text{if } \{u,v\}\subset Y \\
			\infty,& \text{else}
		\end{cases}
		\]
	\end{itemize}
\end{notation}

Now, we will define the metric $d_{\ast}$ on $X\ast Y$, as introduced in \cite{fukaya2023free}(here, we also use the notation $d_{X\ast Y}$ to denote the metric on $X\ast Y$), by considering the distance of the shortest path between two points in $X\ast Y$. A more detailed description follows:

\subsection*{Explanation of the Metric Definition through Examples}

We elucidate the principle for defining the metric $d_{\ast}$ through the following illustrative examples:

		\begin{description}
			\item[\textbf{Example 1}] Consider two points $p$ and $q$ with coordinates $(x_{0}y_{0},u,0)$ and $(\epsilon,v,0)$, respectively, where $u,v\in X$. The shortest path between them is visualized in Figure \ref{figure5}.
			\begin{figure}[htbp]
				\centering
				\includegraphics[width=0.6\textwidth]{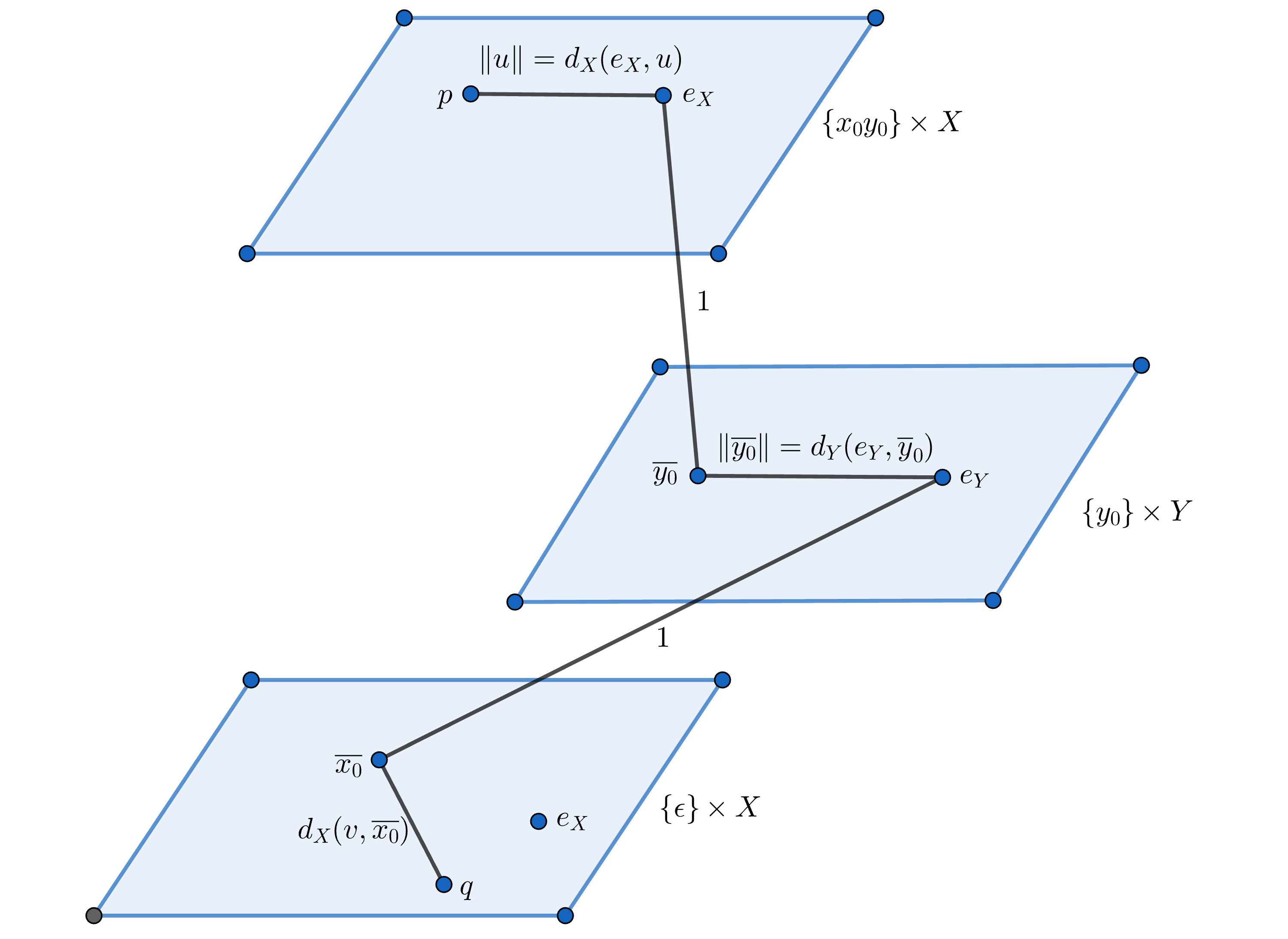}
				\caption{Shortest path between points $p=(x_{0}y_{0},u,0)$ and $q=(\epsilon,v,0)$ as described in Example 1.}
				\label{figure5}
			\end{figure}
			
			We define the metric as $d_{\ast}(p,q) := d_{X}(v,\overline{x_{0}}) + 1 + \Vert \overline{y_{0}} \Vert + 1 + \Vert u \Vert$.
			
			\item[\textbf{Example 2}] Let $p=(\omega,u,0)$ and $q=(\omega x_{0}y_{0}x_{1},v,0)$, where $\omega\in W_{X}$, $x_{i}\in X_{0}^*$, $y_{0}\in Y_{0}^*$, $u\in X$, and $v\in Y$. The shortest path between them is depicted in Figure \ref{figure6}.
			\begin{figure}[htbp]
				\centering
				\includegraphics[width=0.45\textwidth]{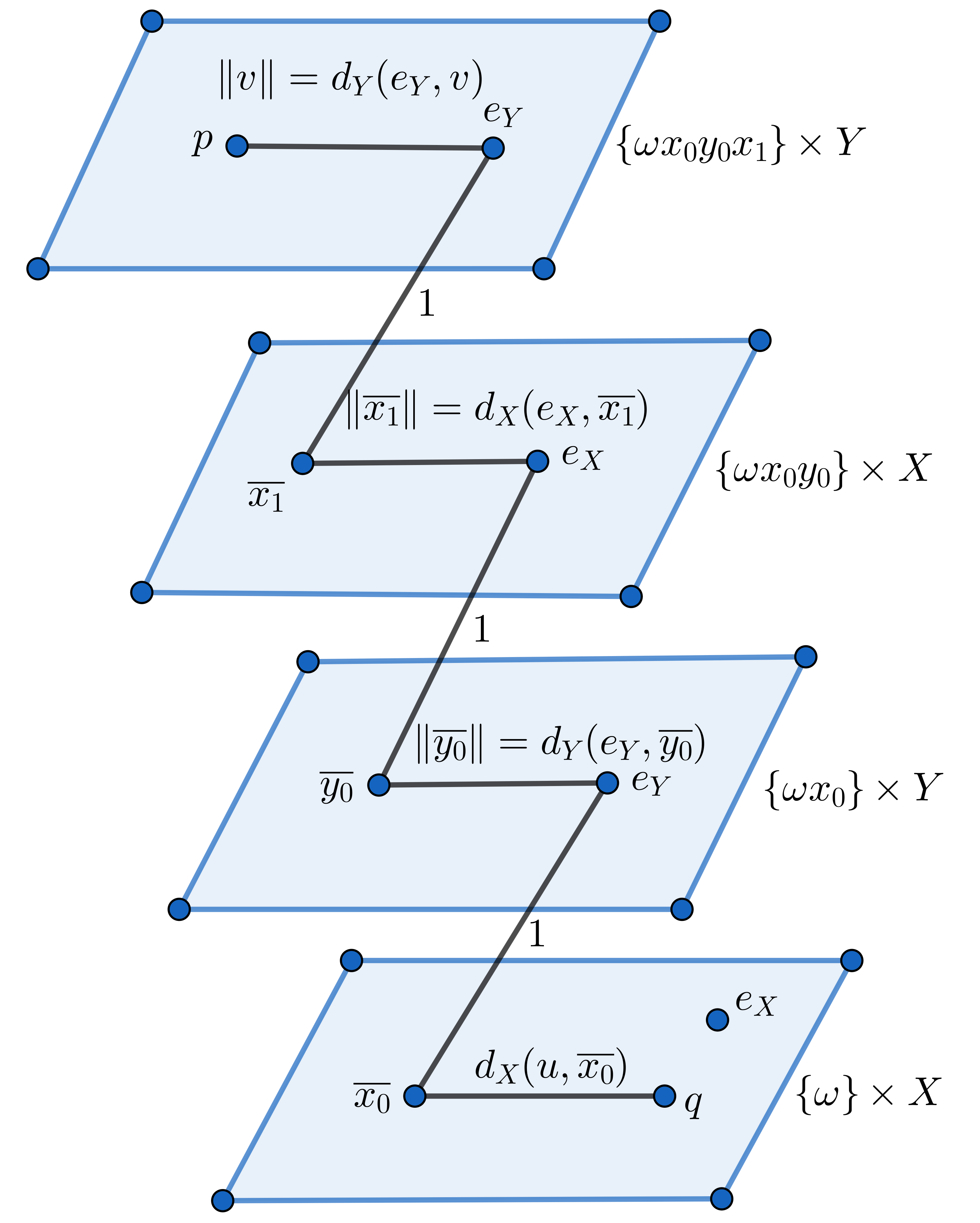}
				\caption{Shortest path between points $p=(\omega,u,0)$ and $q=(\omega x_{0}y_{0}x_{1},v,0)$ as described in Example 2}
				\label{figure6}
			\end{figure}
			
			We define $d_{\ast}(p,q) := \Vert v \Vert + 1 + \Vert \overline{x_{1}} \Vert + 1 + \Vert \overline{y_{0}} \Vert + 1 + d_{X}(u,\overline{x_{0}})$.
			
			\item[\textbf{Example 3}] Consider $p=(\rho x_{0}y_{0},u,0)$ and $q=(\rho x'_{0}y'_{0}x'_{1},v,0)$, where $\rho \in W_{X}$, $x_{0}\ne x'_{0}$, $x_{i},x'_{i}\in X_{0}^*$, $y_{0},y'_{0}\in Y_{0}^*$, $u\in X$, and $v\in Y$. The shortest path between them is illustrated in Figure \ref{figure7}.
			
			\begin{figure}[htbp]
				\centering
				\includegraphics[width=0.7\textwidth]{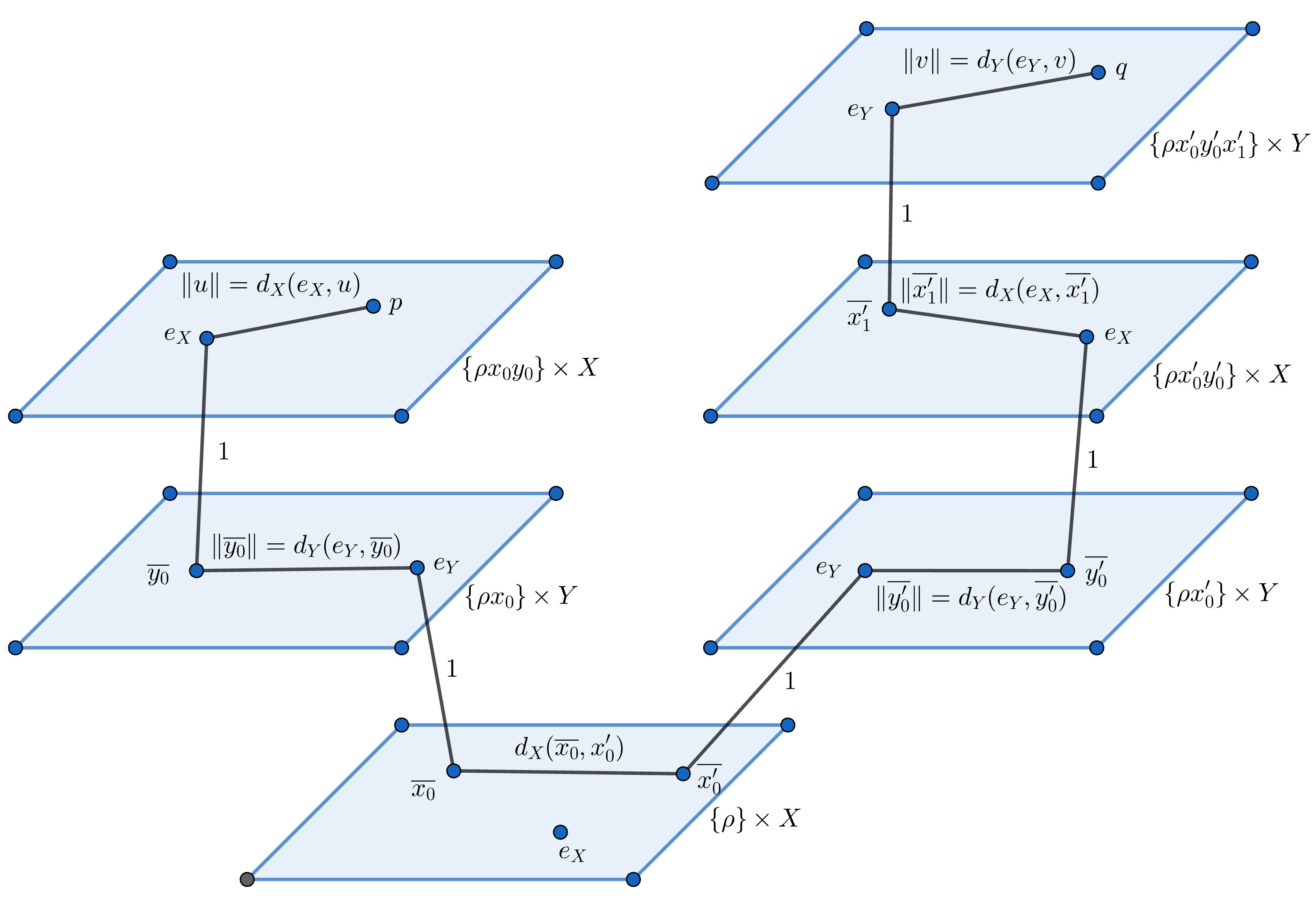}
				\caption{Shortest path between points $p=(\rho x_{0}y_{0},u,0)$ and $q=(\rho x'_{0}y'_{0}x'_{1},v,0)$ as described in Example 3}
				\label{figure7}
			\end{figure}
			We define $d_{\ast}(p,q) := \Vert v \Vert + 1 + \Vert \overline{x'_{1}} \Vert + 1 + \Vert \overline{y'_{0}} \Vert + 1 + d_{X}(\overline{x'_{0}},\overline{x_{0}}) + 1 + \Vert \overline{y_{0}} \Vert + 1 + \Vert u \Vert$.
		\end{description}

\begin{definition}[Construct the Metric $d_{\ast}$ on the Free Product $X\ast Y$]
	The metric $d_{\ast}$ on the free product $X\ast Y$ is constructed in the following steps, as detailed in \cite{fukaya2023free}:
	\begin{description}
		\item[\textbf{Step 1. Defining $d_{\ast}$ on Sheets.}] We define $d_{\ast}$ on sheets following the principles exemplified in the preceding examples. For a rigorous definition, we refer the reader to \cite{fukaya2023free}.
	   \item[\textbf{Step 2. Extending $d_{\ast}$ to Edges via the Interval Metric.}] We extend $d_{\ast}$ to edges by using the standard Euclidean metric on the interval $[0,1]$.  Specifically, for any edge of the form $\{\omega\}\times \{z\}\times [0,1]$ (where $\omega$ is an index word and $z \in X_0$ or $Y_0$ or $\epsilon$), the distance between two points $(\omega, z, t_1)$ and $(\omega, z, t_2)$ on the same edge is given by $|t_1 - t_2|$.
	\end{description}
	From this construction, it is evident that $d_{\ast}$ is symmetric, non-degenerate, and satisfies the triangle inequality. This completes the construction of the metric $d_{\ast}$ on the free product $X\ast Y$.
\end{definition}

The free product of metric spaces, denoted $X \ast Y$, constitutes a refined geometric construction that is analogous to the algebraic free product of groups. Its purpose is to amalgamate metric spaces $(X, d_X)$ and $(Y, d_Y)$ in a manner that embodies a principle of free alternation. This is achieved through the development of a hierarchical, tree-like structure wherein isometric copies of $X$ and $Y$, designated as \textit{sheets}, are interconnected via geodesic segments termed \textit{edges}. Commencing with base sheets that represent $X$ and $Y$, the construction proceeds iteratively by appending subsequent sheets, with the type of sheet (either isometric to $X$ or $Y$) alternating at each juncture. This iterative process effectively realizes the algebraic notion of free alternation within a geometric framework. The resultant structure is a branched metric space wherein geodesic paths are necessarily composed of alternating segments within components derived from $X$ and $Y$, thus reflecting the word structure inherent in algebraic free products.  The significance of this construction within coarse geometry resides in its capacity to potentially preserve essential coarse invariants of the constituent metric spaces and to generate new instances relevant to investigations such as the coarse Baum-Connes conjecture.  It must be noted, however, as exemplified by considerations of rotation groups, that the coarse geometric properties of $X \ast Y$ are not solely determined by $X$ and $Y$ but exhibit a subtle dependence on the selection of \textit{nets} employed in the construction, as shown in Example 1.8 of \cite{fukaya2023free}.

	\subsection*{Free Products of Metric Spaces and Group-Theoretic Free Products:} 
	
	We explore the relationship between categorical constructions in metric geometry and group theory, To formulate this precisely, let $G$ and $H$ be finitely generated groups endowed with word metrics induced by finite generating sets $S_G$ and $S_H$ respectively. Consider the following objects:
	\begin{itemize}
		\item The Cayley graphs $\Gamma(G, S_G)$ and $\Gamma(H, S_H)$ as metric spaces
		\item Their metric free product $\Gamma(G, S_G) \ast \Gamma(H, S_H)$
		\item The Cayley graph of the group free product $G \ast H$ with its natural word metric
	\end{itemize}
	
	More precisely, when regarding the groups as metric spaces through their canonical embeddings $(G, \iota_G, e_G)$ and $(H, \iota_H, e_H)$ where $\iota_G:G\hookrightarrow\Gamma(G,S_G)$ and $\iota_H:H\hookrightarrow\Gamma(H,S_H)$ are inclusion maps of nets, there exist a coarse equivalence between(see \cite{fukaya2023free}):
	\begin{enumerate}
		\item The metric free product space $\Gamma(G, S_G) \ast \Gamma(H, S_H)$ as defined by Pisanski and Tucker \cite{pisanski2002growth}, and
		\item The metric space associated to the group free product $G \ast H$ equipped with the word metric relative to the generating set $S_G \cup S_H$.
	\end{enumerate}

	\section{Construction of a Coarse Embedding}
	
In this section, we first introduce two lemmas that will be used later. Then, by applying the properties of coarse embeddings, we give a direct proof of our first main result, concerning the coarse embedding of free products into Hilbert spaces. Subsequently, we extend this result to uniformly convex Banach spaces by generalizing Day's theorem.
	
	\begin{lemma}\label{normalization}
		Suppose a metric space $X$ is coarsely embeddable into a Hilbert space. Then there exists a coarse embedding $F:X\to \mathcal{H}$ and, for $i=1, 2,$ non-decreasing functions $\rho_{i}:\mathbb{R}_{+}\to \mathbb{R}$ such that:
		\begin{itemize}
			\item[(1)] For all $x, y \in X$, we have $\rho_{1}(d(x,y))\leq \Vert F(x)-F(y)\Vert\leq \rho_{2}(d(x,y))$.
			\item[(2)] $\lim_{t \to+\infty} \rho_{i}(t)=+\infty$, for $i=1, 2$.
			\item[(3)] $\rho_{1}(t)= 1$, for $0<t\le 1$.
		\end{itemize}
	\end{lemma}
	
	\begin{proof}
		Let $\widetilde{F}:X\to \mathcal{H}$ be a coarse embedding, and for $i=1, 2,$ let $\widetilde{\rho_{i}}$ be non-decreasing functions on $\mathbb{R}_{+}$ satisfying:
		\begin{itemize}
			\item[(1)] For all $x, y \in X$, we have $\widetilde{\rho}_{1} (d(x,y))\leq \Vert \widetilde{F}(x)-\widetilde{F}(y)\Vert\leq \widetilde{\rho}_{2}(d(x,y))$.
			\item[(2)] $\lim_{t \to+\infty} \widetilde{\rho}_{i}(t)=+\infty$, for $i=1, 2$.
		\end{itemize}
		We define a new map $F:X\to \mathcal{H}\oplus l^{2}(X)$ by
		\[F(x)=\widetilde{F}(x)\oplus \delta_{x},\]
		where $\delta_{x}$ is the Dirac function at the point $x$. Let
		$$\rho'_{1}(t)= \begin{cases}
			0 , & \text{if } t=0 \\
			\sqrt{(\widetilde{\rho}_{1}(t))^2 + 1}, & \text{if } t>0 \\
		\end{cases}$$
		and
		$$\rho_{2}(t)=\sqrt{(\widetilde{\rho}_{2}(t))^2 + 2}.$$
		It is straightforward to verify that $\rho'_{1}$ and $\rho_{2}$ satisfy conditions (1) and (2). Then, we define
		$$\rho_{1}(t)= \begin{cases}
			0 , & \text{if } t=0 \\
			1  , & \text{if } 0<t\le 1 \\
			\sqrt{(\widetilde{\rho}_{1}(t))^2 + 1}, & \text{if } t>1 \\
		\end{cases}$$
		An example visualization of functions like $\rho_1$ and $\rho'_1$ is provided in Figure \ref{figure8}.
	
		\begin{figure}[htbp]
			\centering
			\includegraphics[width=0.7\textwidth]{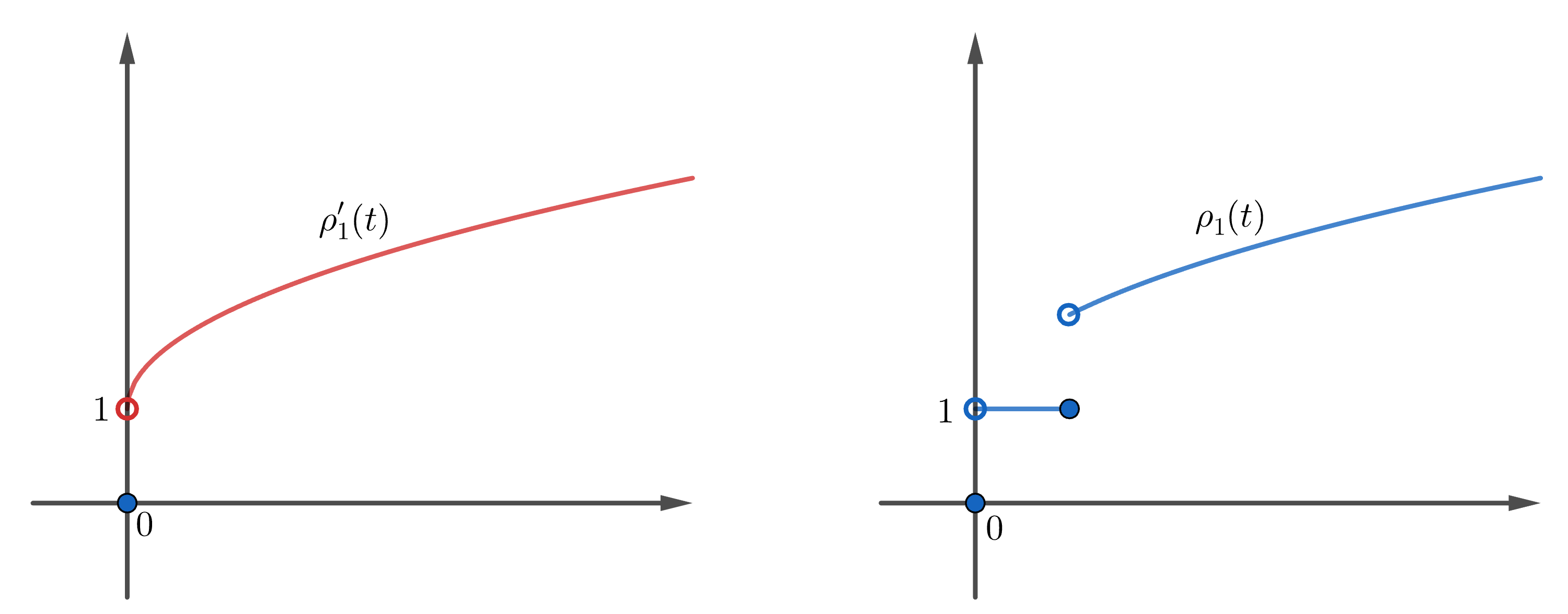}
			\caption{Visualization of functions $\rho'_1(t)$ (red, left) and $\rho_1(t)$ (blue, right). The function $\rho_1(t)$ is constructed to satisfy condition (3) in Lemma \ref{normalization}.}
			\label{figure8}
		\end{figure}
		Clearly, $\rho_1 \leq \rho'_1$. Furthermore, $\rho_1$ and $\rho_2$ satisfy the desired conditions (1), (2), and (3).
	\end{proof}

\begin{lemma}\label{subaddition}
	Let $\rho: \mathbb{R}_{+} \to \mathbb{R}_{+}$ be a non-decreasing function such that $\lim_{t\to +\infty}\rho(t)=+\infty$, with $\rho(t)=1$ for all $t\in (0,1]$ and $\rho(0)=0$. Then, there exists a non-decreasing function $\widetilde{\rho}: \mathbb{R}_{+} \to \mathbb{R}_{+}$ such that $\lim_{t\to +\infty}\widetilde{\rho}(t)=+\infty$, and for all $n\in  \mathbb{N}$ and $(t_{i})_{i=1}^{n}\in \mathbb{R}_{+}^{n}$, we have
	$$\sum_{i=1}^{n}\rho(t_{i})\ge \widetilde{\rho}\left(\sum_{i=1}^{n}t_{i}\right).$$
\end{lemma}

\begin{proof}
	Define $\widetilde{\rho}(t)=\min\left\{\frac{\sqrt{t}}{2},\rho\left(\frac{\sqrt{t}}{2}\right)\right\}$ for $t \in \mathbb{R}_{+}$. Let $n \in \mathbb{N}$ and $(t_{i})_{i=1}^{n}\in \mathbb{R}_{+}^{n}$. Denote $T=\sum_{i=1}^{n}t_{i}$. Without loss of generality, we assume that $t_{i}>0$ for all $1\leq i\le n$.
	
	\textbf{Case 1:}  $n\le \frac{\sqrt{T}}{2}$.
	
	In this case, there exists an index $i_{0}\in \{1, 2, \dots, n\}$ such that $t_{i_{0}}>\frac{\sqrt{T}}{2}$. To see this, suppose for contradiction that $t_{i}\leq \frac{\sqrt{T}}{2}$ for all $i=1, \dots, n$. Then, we would have
	$$\sum_{i=1}^{n}t_{i}\leq \sum_{i=1}^{n}\frac{\sqrt{T}}{2}=n\frac{\sqrt{T}}{2}\leq \frac{\sqrt{T}}{2}\frac{\sqrt{T}}{2}=\frac{T}{4}<T,$$
	which contradicts the definition of $T = \sum_{i=1}^{n}t_{i}$. Therefore, such an index $i_{0}$ must exist. Consequently,
	$$\sum_{i=1}^{n}\rho(t_{i})\geq \rho(t_{i_{0}})\geq \rho\left(\frac{\sqrt{T}}{2}\right)=\rho\left(\frac{\sqrt{\sum_{i=1}^{n}t_{i}}}{2}\right)\geq  \widetilde{\rho}\left(\sum_{i=1}^{n}t_{i}\right).$$
	
	\textbf{Case 2:}  $n> \frac{\sqrt{T}}{2}$.
	
	In this case, we have
	$$\sum_{i=1}^{n}\rho(t_{i})\geq n\cdot 1 > \frac{\sqrt{T}}{2}=\frac{\sqrt{\sum_{i=1}^{n}t_{i}}}{2}.$$
	This inequality holds because $\rho(t)=1$ for all $t\in (0,1]$, and since $t_i > 0$, we have $\rho(t_i) \geq 1$. Thus, using the definition $\widetilde{\rho}(t)=\min\left\{\frac{\sqrt{t}}{2},\rho\left(\frac{\sqrt{t}}{2}\right)\right\}$, we have
	$$\sum_{i=1}^{n}\rho(t_{i}) > \frac{\sqrt{T}}{2} \geq  \widetilde{\rho}(T) = \widetilde{\rho}\left(\sum_{i=1}^{n}t_{i}\right).$$
	Therefore, in both cases, we have shown that $\sum_{i=1}^{n}\rho(t_{i})\ge \widetilde{\rho}\left(\sum_{i=1}^{n}t_{i}\right)$.

\end{proof}

\begin{theorem}[Coarse Embeddability of Free Products]\label{coarse embedd}
	Let $(X,d_{X})$ and $(Y,d_{Y})$ be metric spaces with nets $(X_{0},i_{X},e_{X})$ and $(Y_{0},i_{Y},e_{Y})$ of $X$ and $Y$, respectively, and let $X\ast Y$ be their free product. If both $X$ and $Y$ are coarsely embeddable into a Hilbert space, then $X\ast Y$ is also coarsely embeddable into a Hilbert space.
\end{theorem}

\begin{proof}
	Let $(X,d_{X})$ and $(Y,d_{Y})$ be metric spaces. Assume that $X$ and $Y$ are coarsely embeddable. Let $F_{X}:X\to \mathcal{H}_{X}$ and $F_{Y}:Y\to \mathcal{H}_{Y}$ be coarse embeddings into Hilbert spaces $\mathcal{H}_{X}$ and $\mathcal{H}_{Y}$, respectively. By Lemma \ref{normalization}, there exist non-decreasing functions $\rho_{1,X}, \rho_{2,X}, \rho_{1,Y}, \rho_{2,Y}: \mathbb{R}_{+} \to \mathbb{R}_{+}$ with $\lim_{t\to +\infty}\rho_{i,j}(t)=+\infty$ for $i=1,2$ and $j\in\{X,Y\}$, such that for $i=1,2$, $\rho_{i,X}(t)=1$ and $\rho_{i,Y}(t)=1$ for all $t\in (0,1]$, and $\rho_{i,X}(0)=\rho_{i,Y}(0)=0$. These functions satisfy the conditions for coarse embeddings:
	\begin{equation*}
		\rho_{1,X}(d_{X}(x,x'))\leq \Vert F_{X}(x)-F_{X}(x')\Vert_{\mathcal{H}_{X}}\leq \rho_{2,X}(d_{X}(x,x')) \quad \text{for all } x,x'\in X,
	\end{equation*}
	and
	\begin{equation*}
		\rho_{1,Y}(d_{Y}(y,y'))\leq \Vert F_{Y}(y)-F_{Y}(y')\Vert_{\mathcal{H}_{Y}}\leq \rho_{2,Y}(d_{Y}(y,y')) \quad \text{for all } y,y'\in Y.
	\end{equation*}
	Define $\rho_{1}=\min\{\rho_{1,X} ,\rho_{1,Y}\}$ and $\rho_{2}=\max\{\rho_{2,X} ,\rho_{2,Y}\}$. Then $\rho_{1}$ and $\rho_{2}$ are non-decreasing functions with $\lim_{t\to +\infty}\rho_{i}(t)=+\infty$ for $i=1,2$, and $\rho_{i}(t)=1$ for all $t\in (0,1]$, $\rho_{i}(0)=0$ for $i=1,2$. Consequently, we have
	\begin{align*}
		\rho_{1}(d_{X}(x,x'))&\leq \Vert F_{X}(x)-F_{X}(x')\Vert_{\mathcal{H}_{X}}\leq \rho_{2}(d_{X}(x,x')), \quad \text{for all } x,x'\in X, \\
		\rho_{1}(d_{Y}(y,y'))&\leq \Vert F_{Y}(y)-F_{Y}(y')\Vert_{\mathcal{H}_{Y}}\leq \rho_{2}(d_{Y}(y,y')), \quad \text{for all } y,y'\in Y.
	\end{align*}
	Therefore, without loss of generality, we may assume that there exist non-decreasing functions $\rho_{1}$ and $\rho_{2}$ with $\lim_{t\to +\infty}\rho_{i}(t)=+\infty$ for $i=1,2$, such that $\rho_{i}(t)=1$ for all $t\in (0,1]$, $\rho_{i}(0)=0$ for $i=1,2$, and the coarse embedding conditions above hold.
	
	By adjusting each of $F_{X}$ and $F_{Y}$ by a unitary isomorphism if necessary, we can further assume that $\mathcal{H}_{X}=\mathcal{H}_{Y}=\mathcal{H}$, and that $F_{X}(e_{X})=F_{Y}(e_{Y})=0$, where $e_{X}$ and $e_{Y}$ are base points of metric spaces $X$ and $Y$, respectively.
	
	Define a new Hilbert space $\mathcal{H}_{X\ast Y}$ as the orthogonal direct sum:
	\[
	\mathcal{H}_{X\ast Y}=\left(\bigoplus_{\substack{\text{sheets } s \\ \text{of } X\ast Y}}\mathcal{H}\right) \bigoplus \left(\bigoplus_{\substack{\text{edges } e \\ \text{of } X\ast Y}}\mathbb{R}\right).
	\]
	An element of $\mathcal{H}_{X\ast Y}$ is thus a collection of vectors, with a vector in $\mathcal{H}$ for each sheet of $X\ast Y$ and a real number for each edge of $X\ast Y$. We express an element $p\in \mathcal{H}_{X\ast Y}$ as
	\[
	p=\left(\bigoplus_{\substack{\text{sheet } s \\ \text{of } X\ast Y}}p_{s} \right) \bigoplus \left(\bigoplus_{\substack{\text{edge } e \\ \text{of } X\ast Y}}p_{e}\right).
	\]
	For $p,q\in \mathcal{H}_{X\ast Y}$, the squared norm of the difference is given by
	\[
	\Vert p-q \Vert^{2}= \sum_{\substack{\text{sheets } s \\ \text{of } X\ast Y}}\Vert p_{s}-q_{s} \Vert_{\mathcal{H}}^{2}+ \sum_{\substack{\text{edges } e \\ \text{of } X\ast Y}}(p_{e}-q_{e})^{2}.
	\]
	Next, we define a map $F:X\ast Y\to \mathcal{H}_{X\ast Y}$ which we aim to show is a coarse embedding. Let $(\omega,z,t)\in X\ast Y$, where $\omega \in W$, $z\in X\sqcup Y$, and $t\in [0,1)$. Let $\omega=w_{1}w_{2}\dots w_{n}$. Without loss of generality, assume that $\omega\in W_{X}$ and $w_{1}\in Y_{0}^*$.
	
	For $t>0$, we have $z\in i_{X}(X_{0})\sqcup i_{Y}(Y_{0})$. To define $F(\omega,z,t)\in \mathcal{H}_{X\ast Y}$, we specify its components at each sheet $s$ and edge $e$ of $X\ast Y$. Let $F(\omega,z,t)_{e}$ denote the component of $F(\omega,z,t)$ corresponding to edge $e$, and $F(\omega,z,t)_{s}$ denote the component corresponding to sheet $s$.
	
	\textbf{Definition of Edge Components for t > 0:}
	
	  \textbf{Subcase (a): Lowest Sheet of $(\omega,z,t)$ is of type $X$.}
	$$F(\omega,z,t)_{e}:= \begin{cases}
		t , &  \text{if } e=\{w_{1}w_{2}\dots w_{n}\}\times\{z\}\times [0,1]\\
		1  , &  \text{if } e=\{w_{1}w_{2}\dots w_{k}\}\times\{w_{k+1}\}\times[0,1], \text{ for } 0\le k \le n-1 \\
		0, &  \text{if } e=\{\epsilon\}\times [0,1]\\
		0,& \text{otherwise}.
	\end{cases}$$
	For $k=0$, we interpret $e=\{\epsilon\}\times\{w_{1}\}\times [0,1]$.
	
	   \textbf{Subcase (b): Lowest Sheet of $(\omega,z,t)$ is of type $Y$.}
	$$F(\omega,z,t)_{e}:= \begin{cases}
		t , &  \text{if } e=\{w_{1}w_{2}\dots w_{n}\}\times\{z\}\times [0,1]\\
		1  , &  \text{if } e=\{w_{1}w_{2}\dots w_{k}\}\times\{w_{k+1}\}\times[0,1], \text{ for } 0\le k \le n-1 \\
		1, &  \text{if } e=\{\epsilon\}\times [0,1]\\
		0, & \text{otherwise}.
	\end{cases}$$
	
	Next, we define $F(\omega,z,t)_{s}$ for each sheet $s$.
	$$F(\omega,z,t)_{s}:= \begin{cases}
		F_{X}(z), &  \text{if } s=\{w_{1}w_{2}\dots w_{n}\}\times X \\
		F_{Y}(\overline{w_{n}}), &  \text{if } s=\{w_{1}w_{2}\dots w_{n-1}\}\times Y \\
		\	\vdots & \vdots \\
		F_{X}(\overline{w_{2}}), &  \text{if } s=\{w_{1}\}\times X \\
		F_{Y}(\overline{w_{1}}), &  \text{if } s=\{\epsilon\}\times Y \\
		0, & \text{otherwise}.
	\end{cases}$$
	
	The construction of the map $F$ is conceptually illustrated in Figure \ref{figure9}.
	 \begin{figure}[htbp]
		\centering
		\includegraphics[width=0.7\textwidth]{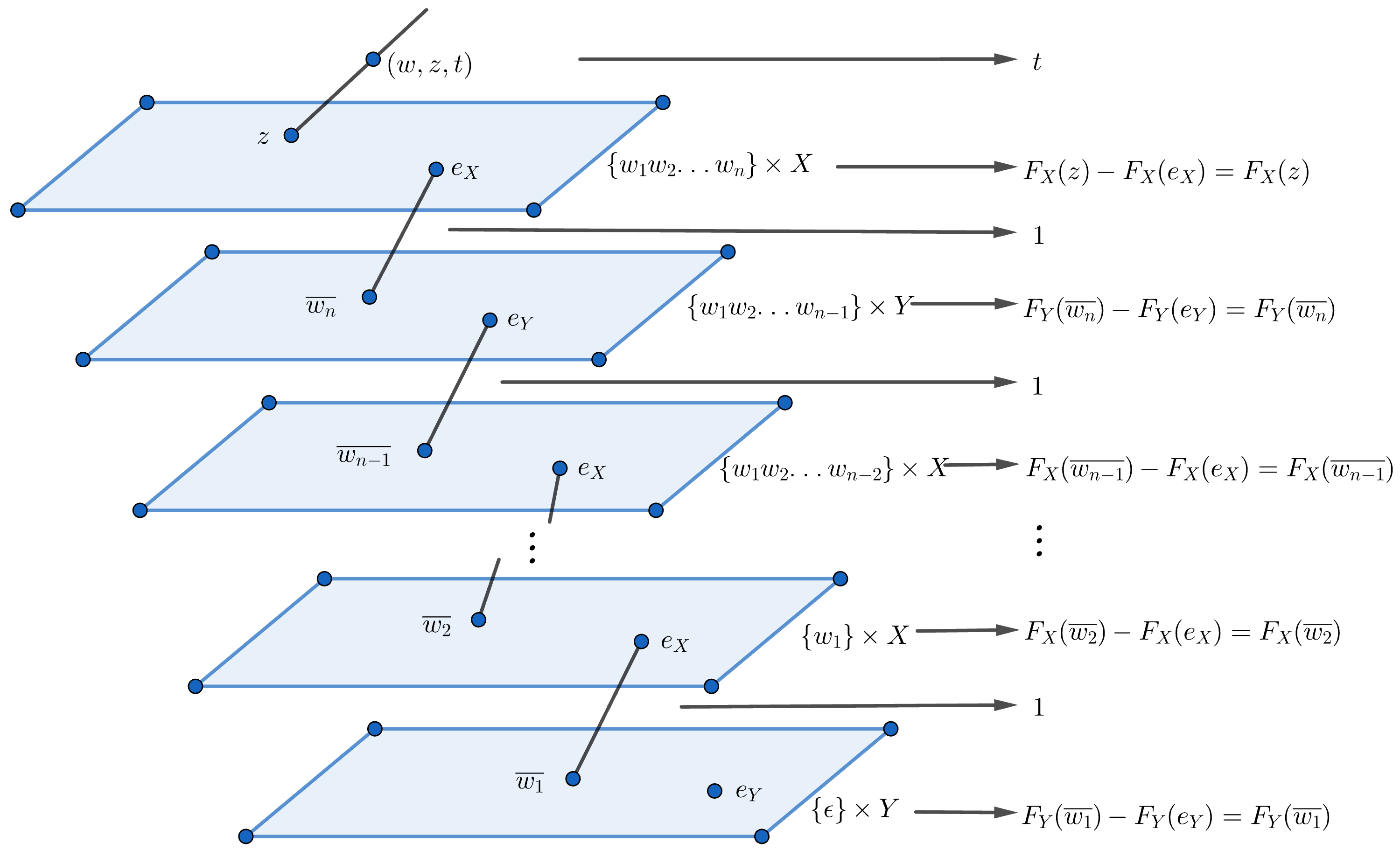}
		\caption{Conceptual illustration of the construction of the coarse embedding $F: X \ast Y \to \mathcal{H}_{X\ast Y}$.}
		\label{figure9}
	\end{figure}
	
	\textbf{Definition of Edge Components for t = 0:} For $t=0$, we have $z\in X\sqcup Y$. The definition of $F(\omega,z,0)$ is derived by setting $t=0$ in the definitions for $t>0$, with adjustments to the edge components $F(\omega,z,0)_{e}$.
	
	  \textbf{Subcase (a): Lowest Sheet of $(\omega,z,0)$ is of type X.}
	$$F(\omega,z,0)_{e}:= \begin{cases}
		1  , &  e=\{w_{1}w_{2}\dots w_{k}\}\times\{w_{k+1}\}\times[0,1], \text{ for } 0\le k \le n-1 \\
		0, &  e=\{\epsilon\}\times [0,1]\\
		0,& \text{otherwise}
	\end{cases}$$
	For $k=0$, we consider $e=\{\epsilon\}\times\{w_{1}\}\times [0,1]$.
	
	   \textbf{Subcase (b): Lowest Sheet of $(\omega,z,0)$ is of type Y.}
	$$F(\omega,z,0)_{e}:= \begin{cases}
		1  , &  e=\{w_{1}w_{2}\dots w_{k}\}\times\{w_{k+1}\}\times[0,1], \text{ for } 0\le k \le n-1 \\
		1, &  e=\{\epsilon\}\times [0,1]\\
		0, & \text{otherwise}.
	\end{cases}$$
	
	The sheet components $F(\omega,z,0)_{s}$ remain as defined for $t>0$:
	$$F(\omega,z,0)_{s}:= \begin{cases}
		F_{X}(z), &  s=\{w_{1}w_{2}\dots w_{n}\}\times X \\
		F_{Y}(\overline{w_{n}}), &  s=\{w_{1}w_{2}\dots w_{n-1}\}\times Y \\
		\	\vdots\\
		F_{X}(\overline{w_{2}}), &  s=\{w_{1}\}\times X \\
		F_{Y}(\overline{w_{1}}), &  s=\{\epsilon\}\times Y \\
		0, & \text{otherwise}.
	\end{cases}$$
	
	It remains to demonstrate that $F$ is indeed a coarse embedding. We will examine different cases to establish the coarse embedding properties, based on the structure of the words $\omega$ and $\omega'$:
	
	\textbf{Demonstrating Coarse Embedding - Different Cases for Points in $X \ast Y$}
	
	\textbf{Case 1: Points represented by words with a common prefix and ending in different types of components.}\label{embeddcase1}
	
	Let $(\omega,z,t), (\omega',z',t')\in X\ast Y$ be such that $\omega=\rho x_{1}y_{1}\dots x_{m}y_{m}$ and $\omega'=\rho x'_{1}y'_{1}\dots x'_{n}y'_{n}x'_{n+1}$.  We evaluate the squared norm of the difference:

	Expanding the squared norm difference $F(\omega,z,t)-F(\omega',z',t')$ directly from the definition of $F$ gives:
	\begin{equation*}
		\begin{aligned}
			\Vert F(\omega,z,t)-F(\omega',z',t')\Vert^2
			&= 2m+(2n+1)+t^2+(t')^2+\Vert F_{X,Y}(z)\Vert^2+\sum_{i=1}^{m}\Vert F_{X,Y}(\overline{y_{i}})\Vert^2 \\
			& \quad +\sum_{i=2}^{m}\Vert F_{X,Y}(\overline{x_{i}})\Vert^2 +\Vert F_{X,Y}(\overline{x_{1}})-F_{X,Y}(\overline{x'_{1}})\Vert^2+\sum_{i=2}^{n+1}\Vert F_{X,Y}(\overline{x'_{i}})\Vert^2 \\
			& \quad +\sum_{i=1}^{n}\Vert F_{X,Y}(\overline{y'_{i}})\Vert^2+\Vert F_{X,Y}(z')\Vert^2.
		\end{aligned}
	\end{equation*}

	Since $t^2 \leq 1$ and $(t')^2 \leq 1$, we bound these terms to simplify the expression:
	\begin{equation*}
		\begin{aligned}
			\Vert F(\omega,z,t)-F(\omega',z',t')\Vert^2
			&\leq 2m+(2n+1)+2+\Vert F_{X,Y}(z)\Vert^2+\sum_{i=1}^{m}\Vert F_{X,Y}(\overline{y_{i}})\Vert^2 \\
			& \quad +\sum_{i=2}^{m}\Vert F_{X,Y}(\overline{x_{i}})\Vert^2 +\Vert F_{X,Y}(\overline{x_{1}})-F_{X,Y}(\overline{x'_{1}})\Vert^2+\sum_{i=2}^{n+1}\Vert F_{X,Y}(\overline{x'_{i}})\Vert^2 \\
			& \quad +\sum_{i=1}^{n}\Vert F_{X,Y}(\overline{y'_{i}})\Vert^2+\Vert F_{X,Y}(z')\Vert^2.
		\end{aligned}
	\end{equation*}

	Using the properties $\Vert F_{X,Y}(u) \Vert^2 \leq \rho_{2}^2(d_{X,Y}(u, e_{X,Y}))$, $\Vert F_{X,Y}(u)-F_{X,Y}(v) \Vert^2 \leq \rho_{2}^2(d_{X,Y}(u,v))$, and $1 \leq \rho_{2}^2(1)$, we get:
	\begin{equation*}
		\begin{aligned}
			\Vert F(\omega,z,t)-F(\omega',z',t')\Vert^2
			&\leq \sum_{j=1}^{2m+(2n+3)} \rho_{2}^2(1) +\rho_{2}^2(d_{X,Y}(z,e_{X,Y}))+\sum_{i=1}^{m}\rho_{2}^2(d_{X,Y}(\overline{y_{i}},e_{X,Y})) \\
			& \quad +\sum_{i=2}^{m}\rho_{2}^2(d_{X,Y}(\overline{x_{i}},e_{X,Y}))+\rho_{2}^2(d_{X,Y}(\overline{x_{1}},\overline{x'_{1}}))+\sum_{i=1}^{n+1}\rho_{2}^2(d_{X,Y}(\overline{x'_{i}},e_{X,Y})) \\
			& \quad +\sum_{i=1}^{n}\rho_{2}^2(d_{X,Y}(\overline{y'_{i}},e_{X,Y}))+\rho_{2}^2(d_{X,Y}(z',e_{X,Y})).
		\end{aligned}
	\end{equation*}
	
By applying the non-decreasing property of $\rho_{2}^2$ and noting that the number of terms in the penultimate equation is bounded by $4d_{X\ast Y}((\omega,z,t),(\omega',z',t'))+4$, we arrive at the concluding inequality:
	\begin{equation*}
		\Vert F(\omega,z,t)-F(\omega',z',t')\Vert^2 \leq \left(4d_{X\ast Y}((\omega,z,t),(\omega',z',t'))+4\right)\rho_{2}^2(d_{X\ast Y}((\omega,z,t),(\omega',z',t'))).
	\end{equation*}
	This provides an upper bound in terms of the distance in $X \ast Y$.

	 Define $\eta_{2}(t)=\sqrt{4t+4}\rho_{2}(t)$. Taking the square root and applying this definition, we directly obtain
	 \begin{equation*}
	 	\Vert F(\omega,z,t)-F(\omega',z',t')\Vert\leq \eta_{2}(d_{X\ast Y}((\omega,z,t),(\omega',z',t'))).
	 \end{equation*}
	 Thus, the map $F$ is bounded by $\eta_{2}$ with respect to the distance $d_{X\ast Y}$.

	To establish a lower bound for the squared norm difference, we begin again from the expanded form:
	\begin{align*}
		\Vert F(\omega,z,t)-F(\omega',z',t')\Vert^2
		&= 2m+(2n+1)+t^2+(t')^2+\Vert F_{X,Y}(z)\Vert^2+\sum_{i=1}^{m}\Vert F_{X,Y}(\overline{y_{i}})\Vert^2 \\
		& \quad +\sum_{i=2}^{m}\Vert F_{X,Y}(\overline{x_{i}})\Vert^2 +\Vert F_{X,Y}(\overline{x_{1}})-F_{X,Y}(\overline{x'_{1}})\Vert^2+\sum_{i=2}^{n+1}\Vert F_{X,Y}(\overline{x'_{i}})\Vert^2 \\
		& \quad +\sum_{i=1}^{n}\Vert F_{X,Y}(\overline{y'_{i}})\Vert^2+\Vert F_{X,Y}(z')\Vert^2.
	\end{align*}
     Since \(t^{2}\geq0\) and \((t')^{2}\geq0\), we can rewrite the above expression as:
	\begin{align*}
		\Vert F(\omega,z,t)-F(\omega',z',t')\Vert^2
		&\ge 2m+(2n+1)+1+1-2+\Vert F_{X,Y}(z)\Vert^2+\sum_{i=1}^{m}\Vert F_{X,Y}(\overline{y_{i}})\Vert^2 \\
		& \quad +\sum_{i=2}^{m}\Vert F_{X,Y}(\overline{x_{i}})\Vert^2 +\Vert F_{X,Y}(\overline{x_{1}})-F_{X,Y}(\overline{x'_{1}})\Vert^2+\sum_{i=2}^{n+1}\Vert F_{X,Y}(\overline{x'_{i}})\Vert^2 \\
		& \quad +\sum_{i=1}^{n}\Vert F_{X,Y}(\overline{y'_{i}})\Vert^2+\Vert F_{X,Y}(z')\Vert^2.
	\end{align*}
	Next, we apply the property that norms are lower bounded by $\rho_{1}^2$ applied to the distance. Using the properties  $\Vert F_{X,Y}(u) \Vert^2 \geq \rho_{1}^2(d_{X,Y}(u, e_{X,Y}))$, $\Vert F_{X,Y}(u)-F_{X,Y}(v) \Vert^2 \geq \rho_{1}^2(d_{X,Y}(u,v))$, $\rho_{1}^2(1) = 1$, and $\rho_{1}^2(t) = 1$ for $t \in (0,1]$. From these, we deduce:
	\begin{align*}
		\Vert F(\omega,z,t)-F(\omega',z',t')\Vert^2
		& \ge \overbrace{\rho_{1}^2(1)+\dots+\rho_{1}^2(1)}^{2m+(2n+1) \text{ items}}+\rho_{1}^2(t)+\rho_{1}^2(t')-2+\rho_{1}^2(d_{X,Y}(z,e_{X,Y})) \\
		& \quad +\sum_{i=1}^{m}\rho_{1}^2(d_{X,Y}(\overline{y_{i}},e_{X,Y})) +\sum_{i=2}^{m}\rho_{1}^2(d_{X,Y}(\overline{x_{i}},e_{X,Y}))+\rho_{1}^2(d_{X,Y}(\overline{x_{1}},\overline{x'_{1}})) \\
		& \quad +\sum_{i=1}^{n+1}\rho_{1}^2(d_{X,Y}(\overline{x'_{i}},e_{X,Y}))+\sum_{i=1}^{n}\rho_{1}^2(d_{X,Y}(\overline{y'_{i}},e_{X,Y}))+\rho_{1}^2(d_{X,Y}(z',e_{X,Y})).
	\end{align*}
	Applying Lemma \ref{subaddition}, which relates the sum of these terms to the distance in $X\ast Y$, we get:
	\begin{align*}
		\Vert F(\omega,z,t)-F(\omega',z',t')\Vert^2
		& \ge \overbrace{\rho_{1}^2(1)+\dots+\rho_{1}^2(1)}^{2m+(2n+1) \text{ items}}+\rho_{1}^2(t)+\rho_{1}^2(t')-2+\rho_{1}^2(d_{X,Y}(z,e_{X,Y})) \\
		& \quad +\sum_{i=1}^{m}\rho_{1}^2(d_{X,Y}(\overline{y_{i}},e_{X,Y})) +\sum_{i=2}^{m}\rho_{1}^2(d_{X,Y}(\overline{x_{i}},e_{X,Y}))+\rho_{1}^2(d_{X,Y}(\overline{x_{1}},\overline{x'_{1}})) \\
		& \quad +\sum_{i=1}^{n+1}\rho_{1}^2(d_{X,Y}(\overline{x'_{i}},e_{X,Y}))+\sum_{i=1}^{n}\rho_{1}^2(d_{X,Y}(\overline{y'_{i}},e_{X,Y}))+\rho_{1}^2(d_{X,Y}(z',e_{X,Y})) \\
		&\ge \widetilde{\rho}(d_{X\ast Y}((\omega,z,t),(\omega',z',t')))-2,
	\end{align*}
	where $\widetilde{\rho}$ is defined as in Lemma \ref{subaddition}, applied with $\rho=\rho_{1}^2$.
	Let us define a function $\zeta(t)=\max\{0,\widetilde{\rho}(t)-2\}$ to ensure non-negativity. Then,
	\begin{align*}
		\Vert F(\omega,z,t)-F(\omega',z',t')\Vert^2
		&\geq \zeta(d_{X\ast Y}((\omega,z,t),(\omega',z',t'))).
	\end{align*}
	Finally, by taking the square root and defining $\eta_{1}(t)=\sqrt{\zeta(t)}$, we obtain the lower bound for the norm:
	\begin{align*}
		\Vert F(\omega,z,t)-F(\omega',z',t')\Vert
		&\geq \eta_{1}(d_{X\ast Y}((\omega,z,t),(\omega',z',t'))).
	\end{align*}
	This configuration is graphically illustrated in Figure \ref{figure10}.
	
	\begin{figure}[htbp]
		\centering
		\includegraphics[width=0.8\textwidth]{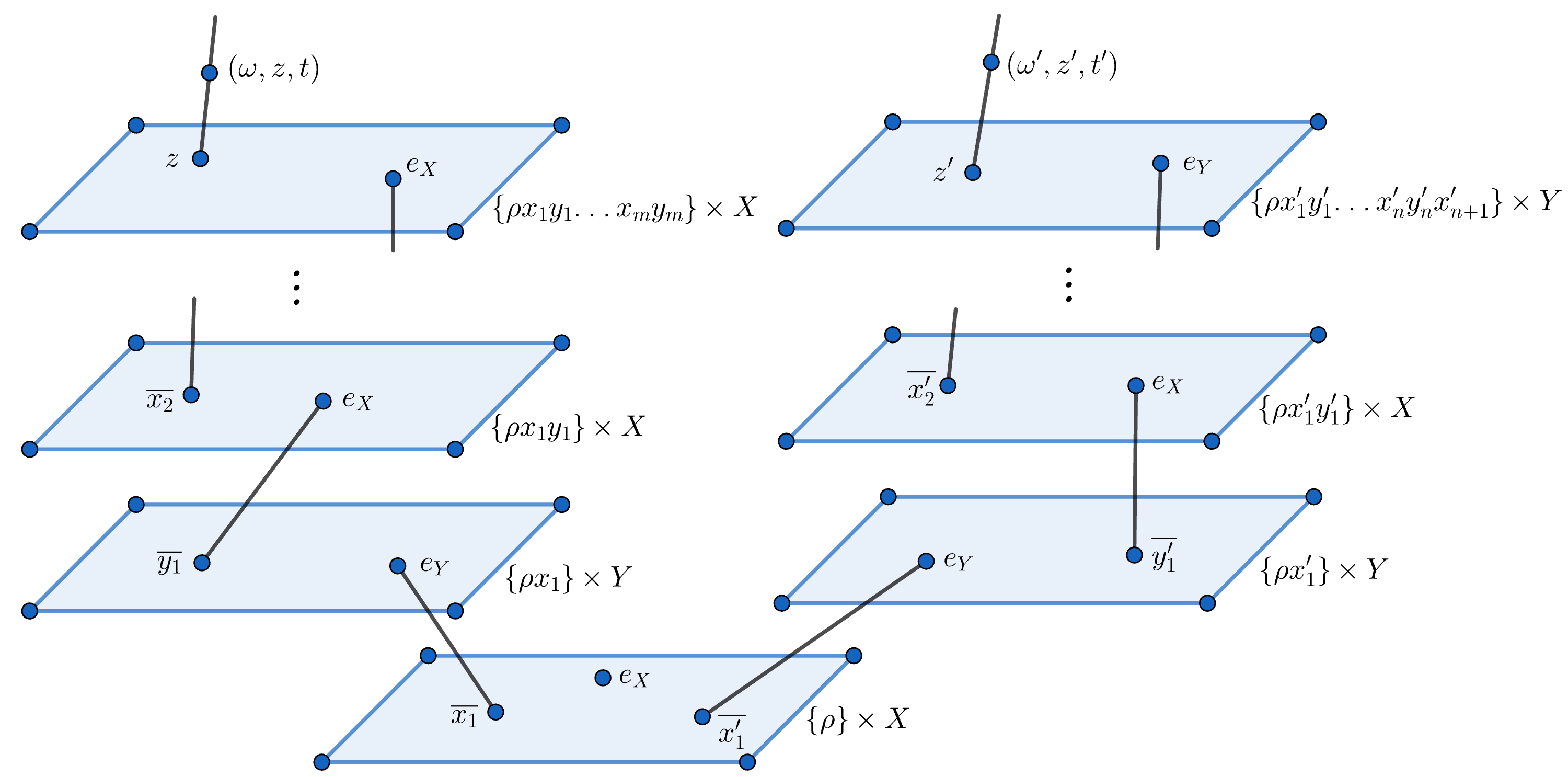}
		\caption{Case $1$ configuration: Points $(\omega, z, t)$ and $(\omega', z', t')$ in $X \ast Y$ represented by words with a common prefix $\rho$ but ending in different types of components.  Here, $\omega = \rho x_1 y_1 \dots x_m y_m$ and $\omega' = \rho x'_1 y'_1 \dots x'_n y'_n x'_{n+1}$.}
		\label{figure10}
	\end{figure}

	\textbf{Case 2: Points represented by words with different starting types.}
	
	Consider points $(\omega,z,t), (\omega',z',t')\in X\ast Y$ where $\omega=x_{1}y_{1}\dots x_{m}y_{m}$ and $\omega'=y'_{1}x'_{2}y'_{2}\dots x'_{n}y'_{n}$.  We begin by expanding the squared norm of the difference of the embeddings:
	\begin{align*}
		\Vert F(\omega,z,t)-F(\omega',z',t')\Vert^2
		&= 2m+1+(2n-1)+t^2+(t')^2+\Vert F_{X,Y}(z)\Vert^2 \\
		&\quad + \sum_{i=1}^{m}\Vert F_{X,Y}(\overline{y_{i}})\Vert^2 + \sum_{i=1}^{m}\Vert F_{X,Y}(\overline{x_{i}})\Vert^2 + \sum_{i=2}^{n}\Vert F_{X,Y}(\overline{x'_{i}})\Vert^2 \\
		&\quad + \sum_{i=1}^{n}\Vert F_{X,Y}(\overline{y'_{i}})\Vert^2 + \Vert F_{X,Y}(z')\Vert^2.
	\end{align*}
	Next, we introduce upper bounds for $t^2$ and $(t')^2$ by 1, since $t, t' \in [0, 1)$. This yields:
	\begin{align*}
		\Vert F(\omega,z,t)-F(\omega',z',t')\Vert^2
		&\leq 2m+1+(2n-1)+1+1+\Vert F_{X,Y}(z)\Vert^2 + \sum_{i=1}^{m}\Vert F_{X,Y}(\overline{y_{i}})\Vert^2 \\
		&\quad + \sum_{i=1}^{m}\Vert F_{X,Y}(\overline{x_{i}})\Vert^2 + \sum_{i=2}^{n}\Vert F_{X,Y}(\overline{x'_{i}})\Vert^2 + \sum_{i=1}^{n}\Vert F_{X,Y}(\overline{y'_{i}})\Vert^2 \\
		&\quad + \Vert F_{X,Y}(z')\Vert^2.
	\end{align*}
	Then, we use the property \(1\leq\rho_{2}^2(1)\) to replace the constant terms. Additionally, we apply the property \(\| F_{X,Y}(\cdot) \|^2\leq\rho_{2}^2(d_{X,Y}(\cdot,e_{X,Y}))\) for the squared norms of the form \(\| F_{X,Y}(\cdot) \|^2\). This leads to:
	\begin{align*}
		\Vert F(\omega,z,t)-F(\omega',z',t')\Vert^2
		&\leq \underbrace{\rho_{2}^2(1)+\dots+\rho_{2}^2(1)}_{(2m+2n+2) \text{ items}} + \rho_{2}^2(d_{X,Y}(z,e_{X,Y})) + \sum_{i=1}^{m}\rho_{2}^2(d_{X,Y}(\overline{y_{i}},e_{X,Y})) \\
		&\quad + \sum_{i=1}^{m}\rho_{2}^2(d_{X,Y}(\overline{x_{i}},e_{X,Y})) + \sum_{i=2}^{n}\rho_{2}^2(d_{X,Y}(\overline{x'_{i}},e_{X,Y})) \\
		&\quad + \sum_{i=1}^{n}\rho_{2}^2(d_{X,Y}(\overline{y'_{i}},e_{X,Y})) + \rho_{2}^2(d_{X,Y}(z',e_{X,Y})).
	\end{align*}
	Finally, by using the non-decreasing property of $\rho_{2}^2$ and considering the relationship between the number of terms and the distance $d_{X\ast Y}((\omega,z,t),(\omega',z',t'))$, we arrive at the inequality:
	\begin{align*}
		\Vert F(\omega,z,t)-F(\omega',z',t')\Vert^2
		&\leq \left(4d_{X\ast Y}((\omega,z,t),(\omega',z',t'))+4\right)\rho_{2}^2(d_{X\ast Y}((\omega,z,t),(\omega',z',t'))).
	\end{align*}
	As in Case 1, defining $\eta_{2}(t)=\sqrt{4t+4}\rho_{2}(t)$, we conclude that
	$$ \Vert F(\omega,z,t)-F(\omega',z',t')\Vert\leq \eta_{2}(d_{X\ast Y}((\omega,z,t),(\omega',z',t'))). $$

	Similarly, for the lower bound:
	\begin{align*}
		\Vert F(\omega,z,t)-F(\omega',z',t')\Vert^2
		&= 2m+1+(2n-1)+t^2+(t')^2+\Vert F_{X,Y}(z)\Vert^2 \\
		&\quad + \sum_{i=1}^{m}\Vert F_{X,Y}(\overline{y_{i}})\Vert^2 + \sum_{i=1}^{m}\Vert F_{X,Y}(\overline{x_{i}})\Vert^2 + \sum_{i=2}^{n}\Vert F_{X,Y}(\overline{x'_{i}})\Vert^2 \\
		&\quad + \sum_{i=1}^{n}\Vert F_{X,Y}(\overline{y'_{i}})\Vert^2 + \Vert F_{X,Y}(z')\Vert^2.
	\end{align*}
	We note that $t, t' \in [0, 1)$, so $t^2 \ge 0$ and $(t')^2 \ge 0$.  Thus, we can write:
	\begin{align*}
		\Vert F(\omega,z,t)-F(\omega',z',t')\Vert^2
		&\ge 2m+2n+1+1-2+\Vert F_{X,Y}(z)\Vert^2 + \sum_{i=1}^{m}\Vert F_{X,Y}(\overline{y_{i}})\Vert^2 \\
		&\quad + \sum_{i=1}^{m}\Vert F_{X,Y}(\overline{x_{i}})\Vert^2 + \sum_{i=2}^{n}\Vert F_{X,Y}(\overline{x'_{i}})\Vert^2 + \sum_{i=1}^{n}\Vert F_{X,Y}(\overline{y'_{i}})\Vert^2 + \Vert F_{X,Y}(z')\Vert^2 - 2.
	\end{align*}
	Now we apply the property that squared norms are lower bounded by $\rho_{1}^2$ applied to the distance, using $\Vert F_{X,Y}(\cdot) \Vert^2 \geq \rho_{1}^2(d_{X,Y}(\cdot, e_{X,Y}))$ and $1 = \rho_{1}^2(1)$ . We have:
	\begin{align*}
		\Vert F(\omega,z,t)-F(\omega',z',t')\Vert^2
		&\ge (2m+2n)\rho_{1}^2(1) + \rho_{1}^2(d_{X,Y}(z,e_{X,Y})) + \rho_{1}^2(d_{X,Y}(z',e_{X,Y})) - 2 \\
		&\quad + \sum_{i=1}^{m}\rho_{1}^2(d_{X,Y}(\overline{y_{i}},e_{X,Y})) + \sum_{i=1}^{m}\rho_{1}^2(d_{X,Y}(\overline{x_{i}},e_{X,Y})) \\
		&\quad + \sum_{i=1}^{n}\rho_{1}^2(d_{X,Y}(\overline{y'_{i}},e_{X,Y})) + \sum_{i=2}^{n}\rho_{1}^2(d_{X,Y}(\overline{x'_{i}},e_{X,Y})).
	\end{align*}

	Applying Lemma \ref{subaddition} with $\rho = \rho_{1}^2$, and letting $\widetilde{\rho}$ be the resulting function, we obtain:
	\begin{align*}
		\Vert F(\omega,z,t)-F(\omega',z',t')\Vert^2
		&\ge \widetilde{\rho}(d_{X\ast Y}((\omega,z,t),(\omega',z',t'))) - 2.
	\end{align*}
	Let $\zeta(t)=\max\{0,\widetilde{\rho}(t)-2\}$. Then $\zeta(t) \ge 0$, and we have
	\begin{align*}
		\Vert F(\omega,z,t)-F(\omega',z',t')\Vert^2
		&\geq \zeta(d_{X\ast Y}((\omega,z,t),(\omega',z',t'))).
	\end{align*}
	Defining $\eta_{1}(t)=\sqrt{\zeta(t)}$, we obtain (as illustrated in Figure \ref{figure11})
	\begin{align*}
		\Vert F(\omega,z,t)-F(\omega',z',t')\Vert
		&\geq \eta_{1}(d_{X\ast Y}((\omega,z,t),(\omega',z',t'))).
	\end{align*}

	\begin{figure}[htbp]
		\centering
		\includegraphics[width=0.7\textwidth]{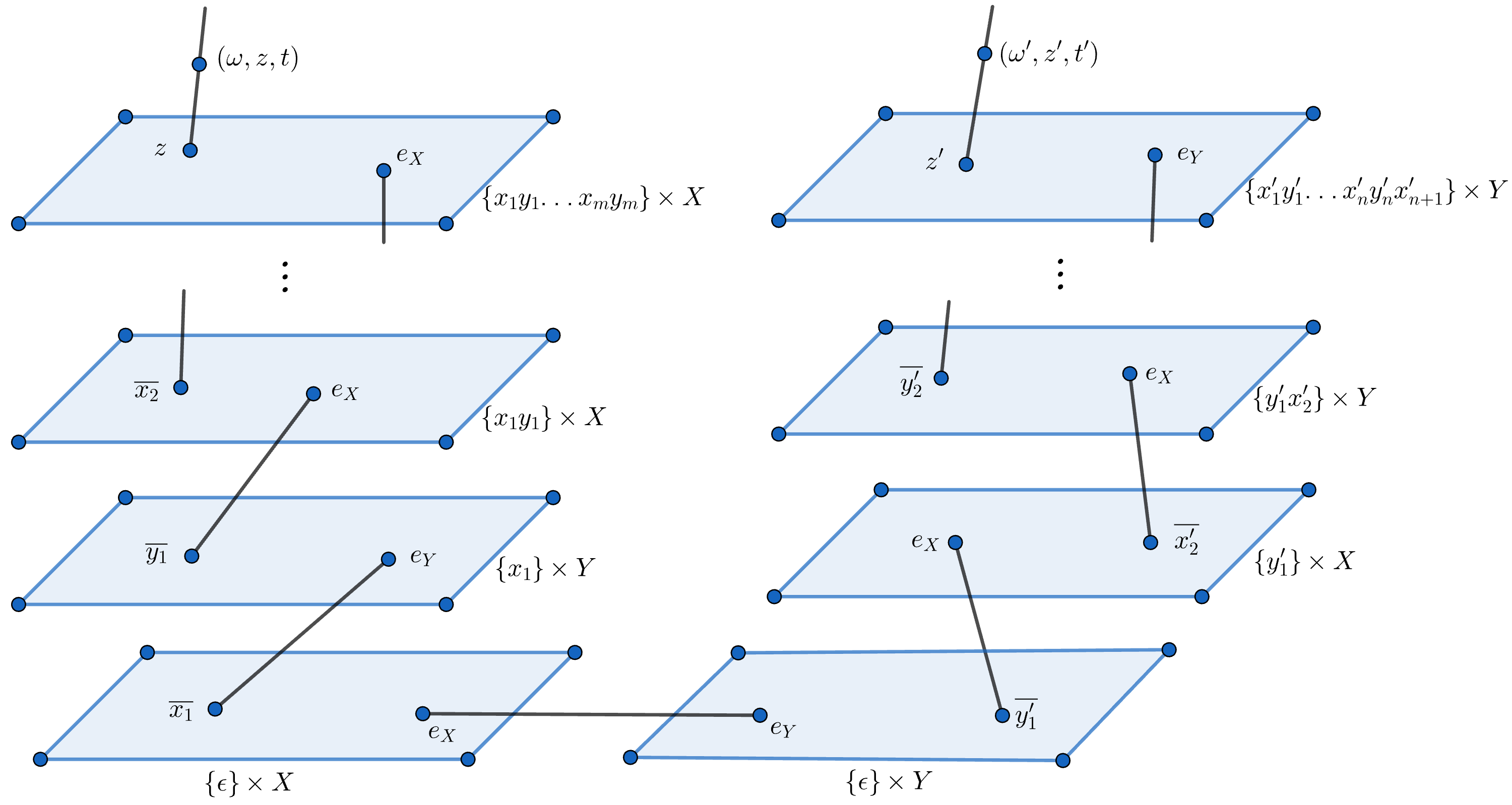}
		\caption{Case $2$ configuration: Points $(\omega, z, t)$ and $(\omega', z', t')$ in $X \ast Y$ represented by words with different starting types.  Here, $\omega=x_{1}y_{1}\dots x_{m}y_{m}$ and $\omega'=y'_{1}x'_{2}y'_{2}\dots x'_{n}y'_{n}$.}
		\label{figure11}
	\end{figure}

	\textbf{Case 3: Points with word indices where one is a prefix of the other.}
	
	Consider $(\omega,z,t),(\omega',z',t')\in X\ast Y$ with $t,t'>0$. Let $\omega=w_{1}w_{2}\dots w_{n}$ and $\omega'=w_{1}w_{2}\dots w_{m}$, and assume $w_{1},w_{n},w_{m}\in X$.
	
	The detailed derivation for this case, being analogous to the preceding two cases, is omitted for brevity. The approach to establish both upper and lower bounds for $\Vert F(\omega,z,t)-F(\omega',z',t')\Vert$ follows a similar line of reasoning as in Case 1 and Case 2, applying the properties of $\rho_1$, $\rho_2$, and Lemma \ref{subaddition}. Thus, it can be concluded that $F$ is a coarse embedding for $X\ast Y$.
\end{proof}

Next, we generalize Theorem~\ref{coarse embedd} to the case of uniformly convex Banach spaces. To this end, we first introduce the concept of coarse embedding a metric space into a Banach space.

\begin{definition}(\cite{gromov1992asymptotic}, 7.E.)

Let $X$ be a metric space, and let $E$ be a Banach space. A map $f : X \longrightarrow E$ is said to be a coarse embedding if there exist non-decreasing functions $\rho_1$ and $\rho_2$ from $\mathbb{R}_+ = [0, \infty)$ to $\mathbb{R}_+$ such that 
\begin{enumerate}
	\item $\rho_1(d(x, y)) \leqslant \| f(x) - f(y)\| \leqslant \rho_2(d(x, y))$ for all $x, y \in \Gamma$; 
	\item $\lim_{r \to \infty} \rho_i(r) = +\infty$ for $i = 1, 2$.
\end{enumerate}

\end{definition}

Using a similar method as in Lemma \ref{normalization}, we can obtain the following lemma:

	\begin{lemma}\label{normalization2}
	Suppose a metric space $X$ is coarsely embeddable into a Banach space. Then there exists a coarse embedding $F:X\to E$ and, for $i=1, 2,$ non-decreasing functions $\rho_{i}:\mathbb{R}_{+}\to \mathbb{R}$ such that:
	\begin{itemize}
		\item[(1)] For all $x, y \in X$, we have $\rho_{1}(d(x,y))\leq \Vert F(x)-F(y)\Vert\leq \rho_{2}(d(x,y))$.
		\item[(2)] $\lim_{t \to+\infty} \rho_{i}(t)=+\infty$, for $i=1, 2$.
		\item[(3)] $\rho_{1}(t)= 1$, for $0<t\le 1$.
	\end{itemize}
   \end{lemma}

  We now proceed to formally introduce uniformly convex Banach spaces.
\begin{definition}(\cite{clarkson1936uniformly})
	A Banach space $E$ is said to be \textit{uniformly convex} if, for every $\varepsilon \in (0,2]$, there is a number $\delta>0$ such that, for all $x,y$ in $E$, the conditions
	\[
	\|x\| = \|y\| = 1\quad\text{and}\quad\|x - y\| \geq \varepsilon
	\]
	imply that
	\[
	\left\|\frac{x + y}{2}\right\| \leq 1 - \delta.
	\]
	The number
	\[
	\delta(\varepsilon) = \inf\left\{1 - \left\|\frac{x + y}{2}\right\|:\|x\| = \|y\| = 1,\,\|x - y\| \geq \varepsilon\right\}
	\]
	is called the \textit{modulus of convexity} of $E$. It is clear that $E$ is uniformly convex if and only if $\delta(\varepsilon)>0$ for every $\varepsilon>0$.
\end{definition}

In the proof of Theorem \ref{coarse embedd}, we constructed a new Hilbert space
\[
\mathcal{H}_{X\ast Y}=\left(\bigoplus_{\substack{\text{sheets } s \\ \text{of } X\ast Y}}\mathcal{H}\right) \bigoplus \left(\bigoplus_{\substack{\text{edges } e \\ \text{of } X\ast Y}}\mathbb{R}\right)
\]
. To extend this to the setting of uniformly convex Banach spaces, we need to consider the direct sum of uniformly convex Banach spaces. This issue has two facets: the countability versus uncountability of the direct sum, and whether the uniformly convex property is preserved under countable and uncountable direct sums. We will first treat the countable case, and subsequently the uncountable case.

Let $\{E_n\}_{n=1}^\infty$ be a sequence of uniformly convex Banach spaces, and let $\delta_n(\varepsilon)$ denote the modulus of convexity of $E_n$ for each $n$ and every $\varepsilon>0$. The sequence $\{E_n\}_{n=1}^\infty$ is said to have a \textit{common modulus of convexity} if there exists a function $\delta(\varepsilon)>0$ such that $\delta_n(\varepsilon)\geq\delta(\varepsilon)$ for all $n$ and all $\varepsilon>0$.
Let $\bigoplus_{n=1}^\infty E_n$ denote the $\ell^2$-direct sum of the sequence $\{E_n\}_{n=1}^\infty$, defined as
\[
\bigoplus_{n=1}^\infty E_n = \left\{ (z_n)_{n=1}^\infty : z_n \in E_n \text{ for each } n, \sum_{n=1}^\infty \|z_n\|^2 < \infty \right\}
\]
equipped with the norm
\[
\|(z_n)_{n=1}^\infty\| = \sqrt{\sum_{n=1}^\infty \|z_n\|^2}.
\]
M. Day \cite{day1941some} proved the following theorem.
\begin{theorem}
	The Banach space $\bigoplus_{n=1}^\infty E_n$ is uniformly convex if and only if the sequence $\{E_n\}_{n=1}^\infty$ has a common modulus of convexity.
\end{theorem}

The following is a generalization of Theorem \ref{coarse embedd} to the countable case.

\begin{corollary}\label{countable}
	Let $(X,d_{X})$ and $(Y,d_{Y})$ be metric spaces with nets $(X_{0},i_{X},e_{X})$ and $(Y_{0},i_{Y},e_{Y})$ of $X$ and $Y$, respectively. Suppose that $X_{0}$ and $Y_{0}$ are countable, and let $X\ast Y$ denote their free product. If both $X$ and $Y$ are coarsely embeddable into a uniformly convex Banach space, then $X\ast Y$ is also coarsely embeddable into a uniformly convex Banach space.
\end{corollary}

 Since the proof of this corollary largely follows the proof of Theorem \ref{coarse embedd}, we will focus on explaining the points of difference. Firstly, since both metric spaces $X$ and $Y$ are coarsely embeddable into a uniformly convex Banach space, we can apply the  Lemma \ref{normalization2} to find common non-decreasing functions $\rho_1$ and $\rho_2$ satisfying conditions (1), (2), and (3) in Lemma \ref{normalization2}.
 
 Next, we need to find a common uniformly convex Banach space $E$ and new (adjusted) coarse embeddings $F'_X: X \to E$ and $F'_Y: Y \to E$ such that:
 \begin{itemize}
 	\item \(E\) is uniformly convex.
 	\item \(F'_X\) and \(F'_Y\) satisfy the inequalities with the same functions $\rho_1$ and $\rho_2$.
 	\item \(F'_X(e_{X}) = 0\) and \(F'_Y(e_{Y}) = 0\), where \(0\) is the zero vector in \(E\).
 \end{itemize}
 
 Unlike Hilbert spaces, any two infinite-dimensional uniformly convex Banach spaces are not necessarily isometrically isomorphic. Therefore, we cannot simply map \(E_Y\) to \(E_X\) via an isometry as before. However, we can construct a new Banach space that contains $E_X$ and $E_Y$ as subspaces and preserves uniform convexity. A standard approach is to consider their $\ell^p$-direct sum. Let us choose $p=2$ and define the space $E = E_X \oplus_2 E_Y$. This space consists of all ordered pairs $(x, y)$, where $x \in E_X$ and $y \in E_Y$, equipped with the norm:
 \[ \|(x, y)\|_E = \left( \|x\|_{E_X}^2 + \|y\|_{E_Y}^2 \right)^{1/2}. \]
 We know that if $E_X$ and $E_Y$ are uniformly convex, then their $\ell^2$-direct sum $E = E_X \oplus_2 E_Y$ is also uniformly convex. (This result follows from Day's Theorem).
 
 Define \(F'_X: X \to E\) as \(F'_X(x) = (F_X(x), 0)\), and \(F'_Y: Y \to E\) as \(F'_Y(y) = (0, F_Y(y))\). Now, the target space for both \(F'_X\) and \(F'_Y\) is the same uniformly convex Banach space \(E\). It is straightforward to verify that the space $E$ and embeddings $F'_X$ and $F'_Y$ constructed in this way satisfy all the desired properties: \(E\) is uniformly convex, \(F'_X\) and \(F'_Y\) satisfy the inequalities with the same functions $\rho_1$ and $\rho_2$, and \(F'_X(e_{X}) = 0\) and \(F'_Y(e_{Y}) = 0\). Next, we define a new  Banach space $E_{X\ast Y}$ as an $\ell_2$-direct sum:
 \[
 E_{X\ast Y}=\left(\bigoplus_{\substack{\text{sheets } s \\ \text{of } X\ast Y}}E\right) \bigoplus \left(\bigoplus_{\substack{\text{edges } e \\ \text{of } X\ast Y}}\mathbb{R}\right).
 \]
 When $X_{0}$ and $Y_{0}$ are countable, the sets of sheets and edges of $X\ast Y$ are countable. This is because if $X_0$ and $Y_0$ are countable, then $X_0^*$ and $Y_0^*$ are also countable, and hence their disjoint union $X_0^* \sqcup Y_0^*$ is countable.
 The set $W$ consists of all finite words over the countable alphabet $X_0^* \sqcup Y_0^*$. It is a well-known fact that the set of all finite words over a countable alphabet is countable. Therefore, $W$ (and its subsets $W_X$ and $W_Y$) are countable. Sheets are defined as sets of the form $\{\omega\}\times X$ and $\{\tau\}\times Y$, where $\omega \in W_{X}$ and $\tau \in W_{Y}$. Since $W_{X}$ and $W_{Y}$ are countable, the collection of sheets is countable. The set of edges is defined as the disjoint union
 \[
 (W_{X}\times X_{0}\times [0,1]) \sqcup (W_{Y}\times Y_{0}\times [0,1]) \sqcup (\epsilon\times [0,1]).
 \]
 From the definition, it is clear that the set of edges is also countable. Thus, by applying Day's Theorem again, we can conclude that $E_{X\ast Y}$ is indeed a uniformly convex Banach space. The subsequent proof is almost identical to the previous one, and we omit the redundant details.
 
 \begin{remark}
 	In Corollary~\ref{countable}, requiring $X_{0}$ and $Y_{0}$ to be countable is not an overly restrictive assumption. If a metric space $X$ is $\sigma$-compact, meaning $X = \bigcup_{i} K_{i}$ (a countable union of compact sets), and the net map $i_{X}: X_{0} \to X$ satisfies that $i_{X}^{-1}(K_{i})$ is finite for each $i$, then $X_{0} = \bigcup_{i} i_{X}^{-1}(K_{i})$ is a countable union of finite sets, implying $X_{0}$ is countable. Many important metric spaces are $\sigma$-compact, for instance, proper metric spaces (such as finite-dimensional, connected and geodesically complete Riemannian manifolds, Cayley graphs of finitely generated groups, and locally finite graphs), and properness is a standard assumption in the context of the coarse Baum-Connes conjecture, as it guarantees desirable properties for the Roe algebra.  Moreover, complete and separable metric spaces, as well as locally compact and separable metric spaces, are also $\sigma$-compact.
 \end{remark}

 We now aim to generalize Theorem \ref{coarse embedd} to the general case, i.e., without assuming that $X_{0}$ and $Y_{0}$ are countable. However, once we drop the assumption of countability for $X_{0}$ and $Y_{0}$, when defining a new uniformly convex Banach space as
 \[
 E_{X\ast Y}=\left(\bigoplus_{\substack{\text{sheets } s \\ \text{of } X\ast Y}}E\right) \bigoplus \left(\bigoplus_{\substack{\text{edges } e \\ \text{of } X\ast Y}}\mathbb{R}\right),
 \]
 this definition involves the direct sum of uncountably many uniformly convex Banach spaces. Therefore, it becomes necessary to consider whether the direct sum of uncountably many uniformly convex Banach spaces remains uniformly convex. Day's theorem only addresses the countable case, and we need to extend Day's theorem to the uncountable case.
 
 \begin{theorem}\label{extend Day}
 	Let $\{B_i\}_{i \in I}$ be a family of Banach spaces, where $I$ is an arbitrary index set (possibly uncountable), and $1 < p < \infty$. Let $B = (\bigoplus_{i \in I} B_i)_{\ell^p}$ denote the $\ell^p$-direct sum space, equipped with the norm
 	\[
 	\|b\| = \left( \sum_{i \in I} \|b_i\|_{B_i}^p \right)^{1/p}.
 	\]
 	If all $B_i$ are uniformly convex Banach spaces and they possess a common modulus of convexity $\delta(\varepsilon)$, then $B$ is uniformly convex.
 \end{theorem}
 
The proof of this theorem is largely analogous to the original proof in the countable case. We omit the specific computations and explain the main ideas and proof strategy. The solid foundation of this extension stems from the inherent constraint in the definition of $\ell^p$-sum spaces: any element $b = (b_i) \in B$ with a finite norm $\|b\| = (\sum_{i \in I} \|b_i\|_{B_i}^p)^{1/p}$ must have a support set $\{i \in I \mid b_i \neq 0\}$ of its non-zero components that is at most countable.
 
 Consequently, for any vectors $b$ and $b'$ chosen in the proof such that $\|b\| = \|b'\| = 1$ and $\|b - b'\| > \varepsilon$, their respective sets of non-zero components and the index subset $I'$ corresponding to their union are all at most countable. This implies that all summation operations in the proof that appear to be over the uncountable set $I$, such as calculating $\|b\|^p$, $\|b'\|^p$, $\|b-b'\|^p$, $\|b+b'\|^p$, and subsequently the norms and distances of the norm sequences $x = \{\|b_i\|\}$, $y = \{\|b_i'\|\}$, like $\|x \pm y\|_{\ell^p(I)}^p$ in the $\ell^p(I)$ space, actually only involve summations over the countable set $I'$, thereby ensuring the well-definedness and validity of these operations.
 
 Based on this crucial observation, both main steps of the original proof can be successfully adapted: Firstly, Case 1 assumes that $\|b_i\| = \|b_i'\| = \beta_i$ holds for all $i$. In this case, the key inequality obtained using the common modulus $\delta$, $\|b_i + b_i'\| < 2(1 - \delta(\gamma_i/\beta_i))\beta_i$, is valid for each component.  Consequently, the summation for $\|b+b'\|$ is performed over the countable set $I'$, and $I'$ is further divided into countable subsets $E$ and $F$ for estimation. Using $\delta(\varepsilon/4)$, effective bounds are obtained for the partial sums. Finally, the connection with $\varepsilon$ is established through $\alpha = (\sum_{i \in E} \beta_i^p)^{1/p}$ (involving the $p$-th power norm of $\|b-b'\|$ and bounds for the sum of $\gamma_i$ over $F$, all of which are carried out on countable sets). The entire derivation process and logical relations rely solely on countable summations and the properties of $\delta$.
 
 Secondly, Case 2 addresses the general case. It initially uses the uniform convexity of the $\ell^p(I)$ space itself (which holds for $1 < p < \infty$, regardless of whether $I$ is countable) and its modulus of convexity $\delta_1$. By assuming $\|b + b'\| > 2(1 - \delta_1(\alpha))$ and combining with the triangle inequality, it derives the distance between norm sequences $\|x - y\|_{\ell^p(I)} < \alpha$ (these norm calculations are also performed on the countable set $I'$).  Next, a crucial step is constructing the auxiliary vector $b'' = \{b_i''\}$. Its construction is component-wise: for indices $i$ where $\|b_i'\| > 0$, we set $b_i'' = \frac{b_i' \|b_i\|}{\|b_i'\|}$; and for indices $i$ where $\|b_i'\| = 0$ (i.e., $b_i' = 0$), we set $b_i'' = b_i$. This precise construction ensures the crucial property $\|b_i''\| = \|b_i\|$ strictly holds for all indices $i$. Simultaneously, this definition directly leads to the required distance estimate $\|b' - b''\| = \|x - y\|_{\ell^p(I)}$ (where $x=\{\|b_i\|\}$ and $y=\{\|b_i'\|\}$) and this distance is known to be less than $\alpha$. Since the original vector $b$ has at most countable support (i.e., only at most countably many $b_i$ are non-zero), and the construction of $b_i''$ depends only on $b_i$ and $b_i'$, the constructed vector $b''$ is also ensured to have at most countable support. Then, applying the triangle inequality $\|b + b''\| \geq \|b + b'\| - \|b' - b''\|$ and appropriately choosing (independent of the cardinality of $I$) $\alpha$ such that $\|b + b''\| > 2(1 - \delta_0(\varepsilon/2))$, we can conclude, because $b$ and $b''$ satisfy the conditions of Case 1 (equal norms, component norms correspondingly equal), and the logic of Case 1 has been confirmed to be extendable, that $\|b - b''\| < \varepsilon/2$. Finally, combining with $\|b' - b''\| < \alpha < \varepsilon/2$, through the triangle inequality $\|b - b'\| \leq \|b - b''\| + \|b'' - b'\|$, the proof is completed for $\|b - b'\| < \varepsilon$.
 
 Therefore, the entire proof process relies merely on operations performable on countable sets, the inherent uniform convexity of the $\ell^p(I)$ space, and the common modulus of convexity $\delta$ shared by the component spaces. It is not hindered by the uncountability of the index set $I$, confirming the universality of the proof.

 \begin{corollary}\label{uncountable}
 	Let $(X,d_{X})$ and $(Y,d_{Y})$ be metric spaces with nets $(X_{0},i_{X},e_{X})$ and $(Y_{0},i_{Y},e_{Y})$ of $X$ and $Y$, respectively. And let $X\ast Y$ denote their free product. If both $X$ and $Y$ are coarsely embeddable into a uniformly convex Banach space, then $X\ast Y$ is also coarsely embeddable into a uniformly convex Banach space.
 \end{corollary}
 
 \begin{proof}
 	The only essential difference from Corollary \ref{countable} lies in the Banach space defined under the current conditions:
 	\[
 	E_{X\ast Y}=\left(\bigoplus_{\substack{\text{sheets } s \\ \text{of } X\ast Y}}E\right) \bigoplus \left(\bigoplus_{\substack{\text{edges } e \\ \text{of } X\ast Y}}\mathbb{R}\right).
 	\]
 	This definition involves the direct sum of uncountably many uniformly convex Banach spaces. By applying Theorem~\ref{extend Day}, we deduce that $E_{X\ast Y}$ remains a uniformly convex Banach space. Here, we employ the norm corresponding to the case $p=2$.  Elsewhere, using the same proof procedure, it can be shown that $X\ast Y$ is also coarsely embeddable into a uniformly convex Banach space.
 	
 \end{proof}

	\section{Property A}
The most central idea in Section 4 is to use the geometric structure of the free product (the Bass-Serre tree) to guide a unified construction, and thus simultaneously prove both coarse embeddability and property A. By adjusting certain details, proofs for both properties become possible. The starting point of Section 4 is the characterization of coarse embeddability and property A from \cite{willett2006some}. These characterizations are both based on Hilbert space valued functions and characterize the properties using distance and inner product conditions. This characterization is key because we can focus on constructing functions that meet specific criteria, which allows for a unified approach. From the definition of the free product of metric spaces, the Bass-Serre tree of the free product naturally arises. The Bass-Serre tree serves as the ``scaffold'' for constructing Hilbert space representations. The tree's structure determines the representation's construction method. Next, we construct the Hilbert space $\mathcal{H}_{X\ast Y}$. Its construction reflects the ``free'' nature of the free product. Each vertex corresponds to a Hilbert space, and the overall space $\mathcal{H}_{X\ast Y}$ freely combines these local spaces. The map $\gamma : X\ast Y\to \mathcal{H}_{X\ast Y}$ is defined to link the metric space free product structure to its geometric position on the Bass-Serre tree. However, it is important to note that the construction of this map actually only relates to the ``sheets'' of the metric space free product, it cannot effectively capture the information of the ``edges''. Therefore, the failure of defining the map directly in this way is expected. For this reason, it is necessary to introduce Lemma \ref{tree} to modify $\gamma$. It's particularly important to point out that the method used in this section benefits from \cite{chen2003uniform}.

\begin{theorem}[Coarse Embeddability]\cite{willett2006some, gromov1992asymptotic}\label{coarse embed chara}
	Let $X$ be a metric space.  $X$ is coarsely embeddable into Hilbert space if and only if for every $\epsilon>0$ and $C>0$, there exists a Hilbert space $\mathcal{H}$ and a function $\xi:X\to \mathcal{H}$ (denoted $(\xi_{x})_{x\in X}$) such that $\Vert\xi_{x}\Vert=1$ for all $x\in X$, and satisfying the following conditions:
	\begin{itemize}
		\item[(1)] $\Vert\xi_{x}-\xi_{x'}\Vert<\epsilon$ if $d(x,x')\le C$.
		\item[(2)] For every $\hat{\epsilon}>0$, there exists $R>0$ such that $|\langle \xi_{x},\xi_{x'} \rangle|<\hat{\epsilon}$ if $d(x,x')\ge R$.
	\end{itemize}
\end{theorem}

\begin{theorem}[Property A]\cite{BrownOzawa2008, willett2006some} \label{property A}
	Let $X$ be a metric space. $X$ has Property A if and only if for every $\epsilon>0$ and $C>0$, there exists a Hilbert space $\mathcal{H}$ and a function $\xi:X\to \mathcal{H}$ (denoted $(\xi_{x})_{x\in X}$) such that $\Vert\xi_{x}\Vert=1$ for all $x\in X$, and satisfying the following conditions:
	\begin{itemize}
		\item[(1)] $\Vert\xi_{x}-\xi_{x'}\Vert<\epsilon$ if $d(x,x')\le C$.
		\item[(2)] There exists $R>0$ such that $\langle \xi_{x},\xi_{x'} \rangle=0$ if $d(x,x')\ge R$.
	\end{itemize}
\end{theorem}
Conditions (1) and (2) are referred to as the \textit{convergence condition} and \textit{support condition}, respectively. While both characterizations share identical requirements regarding unit norm vectors and the convergence condition, they diverge fundamentally in their stipulations for points at large distances. Coarse embeddability merely necessitates that the inner product of the associated Hilbert space vectors becomes arbitrarily small as the distance between points increases. In contrast, Property A imposes a considerably stronger constraint, demanding orthogonality (i.e., a zero inner product) for vectors corresponding to sufficiently distant points. This difference in the support condition reveals a hierarchical relationship between the two properties: Property A can be viewed as a strengthened form of coarse embeddability.

We require some preliminary work before proceeding with the proofs of the theorems establishing coarse embeddability and Property A for free products.

\begin{definition}[Tree]
	A \textit{tree} is a connected graph that contains no cycles.  More precisely, a graph $T = (V, E)$ is a tree if:
	\begin{enumerate}
		\item $T$ is \textit{connected}: For any two vertices $u, v \in V$, there exists a path in $T$ connecting $u$ and $v$.
		\item $T$ is \textit{acyclic}:  $T$ contains no cycles.
	\end{enumerate}
	Here, $V$ is the set of vertices and $E$ is the set of edges, where each edge is associated with a pair of vertices from $V$ (its endpoints).
\end{definition}

For convenience, we define a metric on the edge set $E$, for edges $e, f \in E$, by
\[ d_{T}(e,f) = \text{the number of vertices on the unique geodesic path in } T \text{ connecting } e \text{ and } f. \]

\begin{lemma}[{\cite{chen2003uniform}}]\label{tree}
	Let $T$ be a tree. For every $N\in \mathbb{N}$, there exists a Hilbert space valued function $\tau_{N}:E\to \mathcal{H}$ such that $\Vert \tau_{N}(e)\Vert=1$ for all $e \in E$, and:
	\begin{itemize}
		\item[(1)] If $d_{T}(e,f)\ge 2N$, then $\langle \tau_{N}(e),\tau_{N}(f) \rangle=0$.
		\item[(2)] $\Vert \tau_{N}(e)-\tau_{N}(f)\Vert^2\leq \frac{2}{N}d_{T}(e,f)$.
	\end{itemize}
\end{lemma}

   \begin{definition}[Bass-Serre Tree $T_{X\ast Y}$ of $X\ast Y$]
   	Define the vertex set $V_{X\ast Y}$ of the Bass-Serre tree $T_{X\ast Y}$ associated to $X\ast Y$ to be the set of all sheets of $X\ast Y$, where each sheet is identified with a vertex. The edge set $E_{X\ast Y}$ of the Bass-Serre tree $T_{X\ast Y}$ is defined to be the set of all edges in $X\ast Y$.  If an edge in the free product $X\ast Y$ connects two sheets, then in the Bass-Serre tree $T_{X\ast Y}$, we introduce an edge connecting the vertices corresponding to these two sheets.  The graph $T_{X\ast Y} = (V_{X\ast Y}, E_{X\ast Y})$ constructed in this manner is termed the Bass-Serre tree for $X\ast Y$.
   \end{definition}

  \begin{remark}

 It is not difficult to verify that $T_{X\ast Y}$ is indeed a tree. To show this property, consider any sheet $\omega X$ in $V_{X\ast Y}$. By construction, there exists a unique path from $\omega X$ to a ``lowest level'' sheet of the form $\epsilon Z$, where $Z \in \{X, Y\}$. If $T_{X\ast Y}$ were to contain a cycle, then without loss of generality, we assume that the sheet $\omega X$ is an element of this cycle. Under such conditions, there would exist at least two distinct paths from $\omega X$ to a ``lowest level'' sheet, contradicting the uniqueness of the path.  The connectedness of $T_{X\ast Y}$ is evident, as every sheet $\omega Z'$ (where $Z' \in \{X, Y\}$ and $\omega \in W$) possesses a path connecting it to either $\epsilon X$ or $\epsilon Y$, and $\epsilon X$ and $\epsilon Y$ are themselves connected by an edge.
 
 \end{remark}

 We now present an alternative proof of Theorem \ref{coarse embedd}. Recall that Theorem \ref{coarse embedd} concludes that if metric spaces $X$ and $Y$ are coarsely embeddable into Hilbert space, then their free product $X \ast Y$ is also coarsely embeddable.
 
\begin{proof}

 Let $\epsilon>0$ and $C>0$ be given. By the characterization of coarse embeddability \ref{coarse embed chara}, there exists a Hilbert space $\mathcal{H}$ and a function $\alpha: X\to \mathcal{H}$ with $\Vert \alpha_{x}\Vert=1$ for all $x \in X$, satisfying the following conditions:
 \begin{itemize}
 	\item[(1)] $\Vert\alpha_{x}-\alpha_{x'}\Vert<\frac{\epsilon}{4C+16}$ if $d_{X}(x,x')\le C$.
 	\item[(2)] For every $\hat{\epsilon}>0$, there exists $R>0$ such that $|\langle \alpha_{x},\alpha_{x'} \rangle|<\hat{\epsilon}$ if $d_{X}(x,x')\ge R$.
 \end{itemize}
 Similarly, we obtain a Hilbert space $\mathcal{H}_{Y}$ and a function $\beta: Y\to \mathcal{H}_{Y}$ with $\Vert \beta_{y}\Vert=1$ for all $y \in Y$, satisfying analogous conditions:
 \begin{itemize}
 	\item[(1)] $\Vert\beta_{y}-\beta_{y'}\Vert<\frac{\epsilon}{4C+16}$ if $d_{Y}(y,y')\le C$.
 	\item[(2)] For every $\hat{\epsilon}>0$, there exists $R>0$ such that $|\langle \beta_{y},\beta_{y'} \rangle|<\hat{\epsilon}$ if $d_{Y}(y,y')\ge R$.
 \end{itemize}
 Without loss of generality, applying a suitable unitary operator $\mathcal{H}_{Y}\to \mathcal{H}$, we may assume that $\mathcal{H}_{Y} = \mathcal{H}$ and that $\beta: Y\to \mathcal{H}$ satisfies $\alpha_{e_{X}}=\beta_{e_{Y}}$, i.e., $\alpha({e_{X}})=\beta({e_{Y}})$.
 We consider $\mathcal{H}$ as a pointed Hilbert space with distinguished vector $\omega=\alpha_{e_{X}}=\beta_{e_{Y}}$ (i.e., $\omega=\alpha({e_{X}})=\beta({e_{Y}})$).
 
 Based on the Bass-Serre tree $T_{X\ast Y}$ of $X\ast Y$, we construct a Hilbert space $\mathcal{H}_{X\ast Y}$. Recall that the vertex set of $T_{X\ast Y}$ is $V_{X\ast Y}=W_{X} \sqcup W_{Y}$, and its edge set is $E_{X\ast Y}=W$. For each vertex $v \in V_{X\ast Y}$, we associate the Hilbert space $\mathcal{H}_{v}=\mathcal{H}$, and define the Hilbert space $\mathcal{H}_{X\ast Y}$ as the completion of the algebraic direct limit:
 \[
 \mathcal{H}_{X\ast Y} = \varinjlim_{F \subseteq V_{X\ast Y} \text{ finite}} \left( \bigotimes_{v\in F} \mathcal{H}_{v} \right).
 \]
 Here, for finite subsets $F \subset G \subseteq V_{X\ast Y}$, the map $\bigotimes_{v\in F}\mathcal{H}_{v}\to \bigotimes_{v\in G}\mathcal{H}_{v}$ is defined by tensoring with the distinguished vector $\omega$ for each $v\in G\setminus F$. These maps are isometries, endowing the algebraic direct limit with a natural inner product space structure. $\mathcal{H}_{X\ast Y}$ is then the Hilbert space completion of this direct limit.
 
 For notational convenience, we consider formal infinite tensor products $\rho=\bigotimes_{v\in V_{X\ast Y}}\rho_{v}$, where all but finitely many components $\rho_{v}$ are equal to the distinguished vector $\omega$, as elements of $\mathcal{H}_{X\ast Y}$. Such elements span a dense linear subspace. If $\eta = \bigotimes_{v\in V_{X\ast Y}}\eta_{v}$ is another such element, then their inner product and the norm of their difference are given by:
 \begin{align*}
 	\langle \rho,\eta \rangle &= \prod_{v\in V_{X\ast Y}}\langle \rho_{v},\eta_{v} \rangle_{\mathcal{H}_{v}}, \\
 	\Vert \rho-\eta\Vert^2 &\leq \sum_{v\in V_{X\ast Y}}\Vert \rho_{v}-\eta_{v}\Vert_{\mathcal{H}_{v}}^2,
 \end{align*}
 where in the norm inequality, we assume $\Vert \rho_{v}\Vert, \Vert \eta_{v}\Vert \le 1$. Note that in these expressions, all but finitely many terms in the infinite product are $1$, and in the infinite sum are $0$.
 
 We define a map $\gamma : X\ast Y\to \mathcal{H}_{X\ast Y}$ by the formal infinite tensor product expression:
 \[
 \gamma(\omega,z,t)=\bigotimes_{v\in V_{X\ast Y}}\gamma(\omega,z,t)_{v},
 \]
 where the component $\gamma(\omega,z,t)_{v}$ of $\gamma(\omega,z,t)$ at vertex $v$ is defined as:
 \[
 \gamma(\omega,z,t)_{v}:=
 \begin{cases}
 	\alpha(z) \text{ or } \beta(z), &  \text{if } v \text{ is a sheet } \{w_{1}w_{2}\dots w_{n}\}\times X \text{ or } \{w_{1}w_{2}\dots w_{n}\}\times Y, \\
 	\alpha(\overline{w_{n}}) \text{ or } \beta(\overline{w_{n}}), &  \text{if } v \text{ is a sheet } \{w_{1}w_{2}\dots w_{n-1}\}\times X \text{ or } \{w_{1}w_{2}\dots w_{n-1}\}\times Y, \\
 	\vdots & \vdots \\
 	\alpha(\overline{w_{2}}) \text{ or } \beta(\overline{w_{2}}), &  \text{if } v \text{ is a sheet } \{w_{1}\}\times X \text{ or } \{w_{1}\}\times Y, \\
 	\alpha(\overline{w_{1}}) \text{ or } \beta(\overline{w_{1}}), &  \text{if } v \text{ is a sheet } \{\epsilon\}\times X \text{ or } \{\epsilon\}\times Y, \\
 	\omega, & \text{otherwise}.
 \end{cases}
 \]
 Here, ``$\alpha(z) \text{ or } \beta(z)$ '' indicates that we use $\alpha(z)$ if $v$ is an X-sheet and $\beta(z)$ if $v$ is a Y-sheet, and similarly for $\alpha(\overline{w_{n}}) \text{ or } \beta(\overline{w_{n}})$, and so on. In order to enhance comprehension of the aforementioned definitions, let us consider a normal word $\omega = w_1 w_2 \dots w_n$, and assume that $w_1, w_n \in X_0^*$. In this case, it is helpful to interpret the expression for $\gamma(\omega, z, t)$ as part of the Bass-Serre tree of $X \ast Y$(see Figure \ref{figure13}).

 \begin{figure}[htbp]
 	\centering
 	\includegraphics[width=0.9\textwidth]{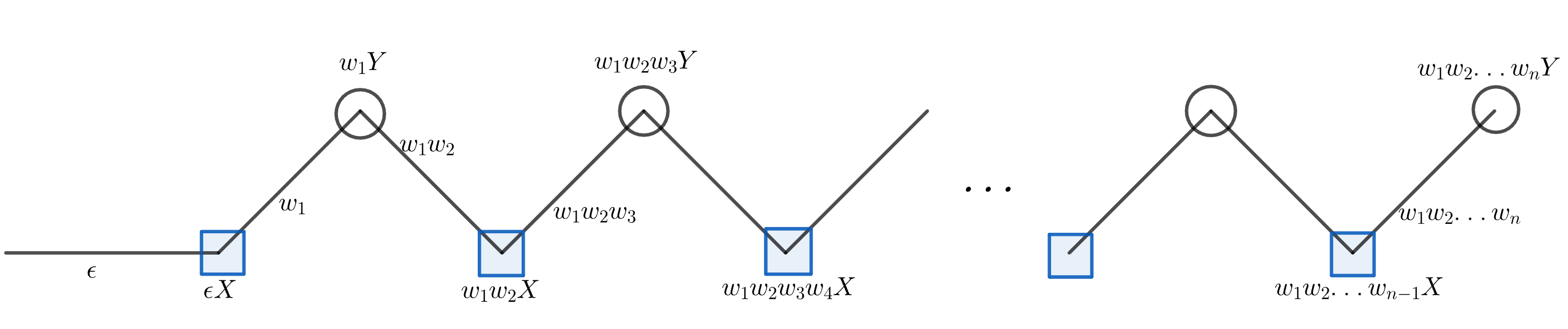}
 	\caption{Bass-Serre tree of $X \ast Y$ illustrating a path corresponding to a normal word $\omega = w_1 w_2 \dots w_n$. The path alternates between X-sheets (squares) and Y-sheets (circles), with edges labeled by prefixes of the word $\omega$ for convience.}
 	\label{figure13}
 \end{figure}
 
 For the vertices along the path from $\epsilon$ to $w_1 w_2 \dots w_n$, and for the vertex $w_1 w_2 \dots w_n Y$, the corresponding tensor for $\gamma(\omega, z, t)$ at these vertices is:
 \[
 \alpha(\overline{w_{1}})\otimes \beta(\overline{w_{2}})\otimes \dots \otimes \alpha(\overline{w_{n}})\otimes\beta(z) \in \mathcal{H}_{\epsilon X}\otimes \mathcal{H}_{w_{1}Y}\otimes\dots \otimes\mathcal{H}_{w_{1}w_{2}\dots w_{n-1}X}\otimes \mathcal{H}_{w_{1}w_{2}\dots w_{n}Y}.
 \]

 For all other vertices, we have $\gamma(\omega, z, t)_v = \omega$. Similarly, for the case where at least one of $w_1, w_n$ is in $Y_0^*$, we can define it analogously. The function $\gamma$ defined above generally does not satisfy the support condition, so we need to make some modifications.
 
 Let $N\ge 32(C+1)\epsilon^{-2}$. Obtain a function $\tau=\tau_{N}:E_{X\ast Y}\to \mathcal{H}$ from Lemma \ref{tree}. Define $\xi :X\ast Y\to \mathcal{H}_{X\ast Y}\otimes \mathcal{H}$ by
 \begin{equation}\label{eq:xi_definition}
 	\xi(\omega,z,t)=\gamma(\omega,z,t)\otimes \tau(\omega).
 \end{equation}
 It is obvious that $\Vert \xi(\omega,z,t)\Vert=1$. We now proceed to verify the convergence and support conditions.

 $\bullet$ \textbf{Support Property}

Let $\hat{\epsilon}>0$ be given. Choose $R>0$ sufficiently large such that conditions (2) hold for both functions $\alpha$ and $\beta$. Consider two points $(\omega,z,t),(\omega',z',t')\in X\ast Y$ with $d_{X\ast Y}((\omega,z,t),(\omega',z',t'))\ge R(2N+2)+2N+4$. We aim to demonstrate that the inner product
$$ \langle \xi(\omega,z,t),\xi(\omega',z',t') \rangle=\langle \gamma(\omega,z,t),\gamma(\omega',z',t') \rangle \langle \tau(\omega),\tau(\omega') \rangle $$
has an absolute value less than $\hat{\epsilon}$.  Observe that each factor in this product is bounded in absolute value by $1$.  Referring to Lemma \ref{tree}, if $d_{T_{X\ast Y}}(\omega,\omega')\ge 2N$, then $\langle \tau(\omega),\tau(\omega') \rangle=0$. Consequently, it is sufficient to show that $\langle \gamma(\omega,z,t),\gamma(\omega',z',t') \rangle$ has absolute value less than $\hat{\epsilon}$ under the assumption that $d_{T_{X\ast Y}}(\omega,\omega')<2N$. We proceed to analyze this in several cases.

 \textbf{Case 1:} Consider $(\omega,z,t),(\omega',z',t')\in X\ast Y$ with $t,t'>0$. Let us express $\omega$ and $\omega'$ as normal words with a common prefix $h \neq \epsilon$. Suppose $\omega=hw_{0}w_{1}w_{2}\dots w_{n}$ and $\omega'=hw'_{0}w'_{1}w'_{2}\dots w'_{m}$, where $w_{0},w_{0}'\in X_{0}^*$ and $w_{n},w'_{m} \in X_{0}^*$. Then the inner product is given by:

\begin{align}\label{case1}
	\langle \gamma(\omega,z,t),\gamma(\omega',z',t') \rangle=
	&\left(\prod_{i=1}^{n}\langle \alpha\beta(\overline{w_{i}}),\omega\rangle\right) \langle\alpha\beta(\overline{w_{0}}),\alpha\beta(\overline{w'_{0}}) \rangle  \left(\prod_{j=1}^{m}\langle\omega,\alpha\beta(\overline{w'_{j}}) \rangle\right)\\
	\notag&\left(\langle\alpha\beta(z),\omega \rangle\right)\left(\langle\alpha\beta(z'),\omega \rangle\right)
\end{align}

Here, $\alpha\beta$ in the formula should be interpreted as either $\alpha$ or $\beta$, chosen appropriately according to whether the element is from $X$ or $Y$. To better understand this, consider Figure \ref{figure14}, which depicts a portion of the Bass-Serre tree.
\begin{figure}[htbp]
	\centering
	\includegraphics[width=0.9\textwidth]{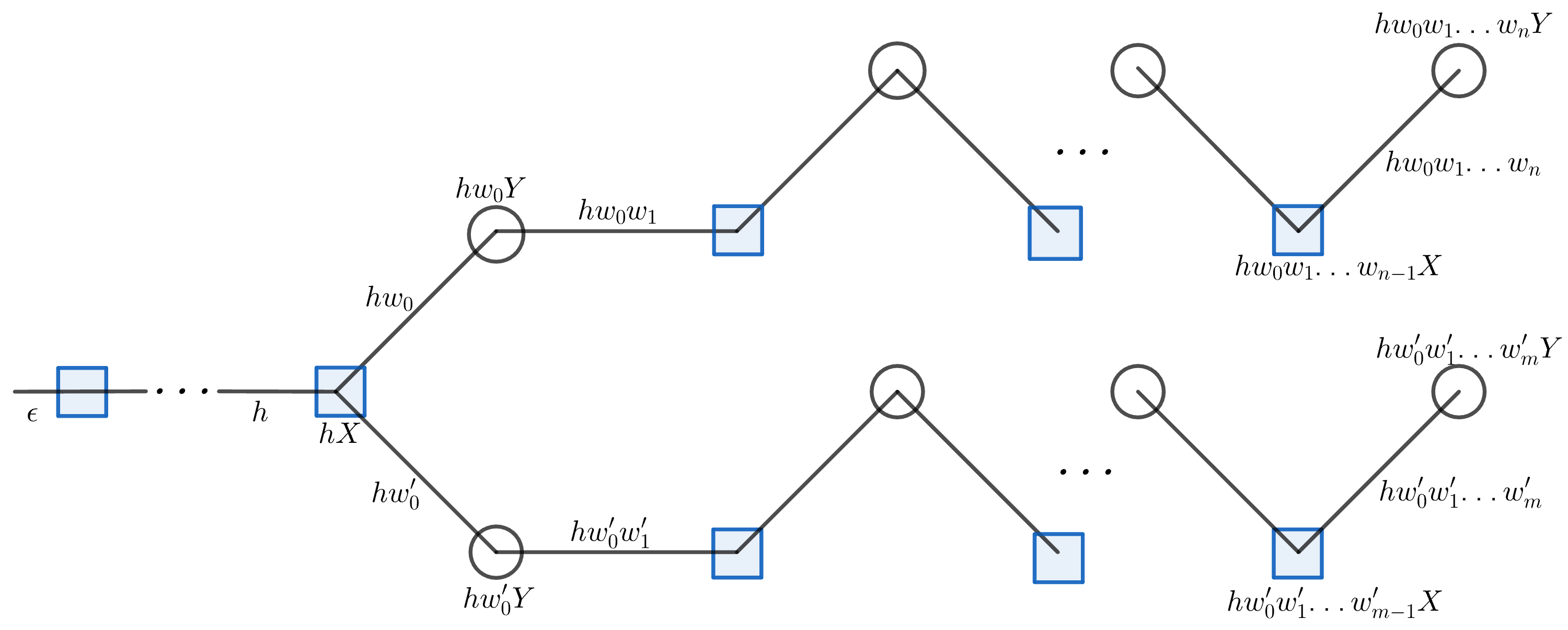}
	\caption{Bass-Serre tree illustrating Case 1 configuration.  Points $(\omega, z, t)$ and $(\omega', z', t')$ share a common prefix $h$ in their normal word representations, where $\omega = hw_0 w_1 \dots w_n$ and $\omega' = hw'_0 w'_1 \dots w'_m$.}
	\label{figure14}
\end{figure}
As illustrated in Figure \ref{figure14}, we observe that the components of $\gamma(\omega,z,t)$ and $\gamma(\omega',z',t')$ are identical for all vertices except those located on the paths from $hX$ to $hw_{0}w_{1}w_{2}\dots w_{n}Y$ and from $hX$ to $hw'_{0}w'_{1}w'_{2}\dots w'_{m}Y$. Specifically, at the vertex $hX$, the components are $\alpha(\overline{w_{0}})$ and $\alpha(\overline{w'_{0}})$ for $\gamma(\omega,z,t)$ and $\gamma(\omega',z',t')$ respectively, and these contribute the term $\langle\alpha\beta(\overline{w_{0}}),\alpha\beta(\overline{w'_{0}}) \rangle$ to the inner product. Similarly, at the vertex $hw_{0}Y$, the respective components are $\beta(w_{1})$ and $\omega$, which contribute the term $\langle\alpha\beta(\overline{w_{1}}),\omega \rangle $.  Analogous contributions arise from the components at other vertices along these paths.

Since every term in the product \eqref{case1} has an absolute value no greater than $1$, our strategy is to demonstrate that at least one term must have an absolute value less than $\hat{\epsilon}$. We will proceed by proof by contradiction. We proceed by contradiction. Assume the opposite, i.e., that every term in product \eqref{case1} has an absolute value greater than or equal to $\hat{\epsilon}$. Since metric spaces $X$ and $Y$ are coarsely embeddable, this assumption implies the following bounds on distances within $X$ and $Y$: $d_{X,Y}(\overline{w_{i}},e_{X,Y})<R$ for $1\leq i\leq n$; $d_{X,Y}(\overline{w_{0}},\overline{w'_{0}})<R$; $d_{X,Y}(e_{X,Y},\overline{w'_{j}})<R$ for $1\leq j\leq m$; $d_{X,Y}(e_{X,Y},z)<R$; and $d_{X,Y}(e_{X,Y},z')<R$. Under these conditions, we would have the following inequality for the distance $d_{X\ast Y}((\omega,z,t),(\omega',z',t'))$:
$$d_{X\ast Y}((\omega,z,t),(\omega',z',t'))\leq 4+d_{T_{X\ast Y}}(\omega,\omega')+R(d_{T_{X\ast Y}}(\omega,\omega')+2).$$

This inequality arises because the number of edges in the path between $(\omega,z,t)$ and $(\omega',z',t')$ is at most $4+d_{T_{X\ast Y}}(\omega,\omega')$, and the distances within sheets are bounded by $R(d_{T_{X\ast Y}}(\omega,\omega')+2)$. Since we have assume that $d_{T_{X\ast Y}}(\omega,\omega')<2N$, we obtain
\begin{align*}
	d_{X\ast Y}((\omega,z,t),(\omega',z',t'))
	&\leq 4+d_{T_{X\ast Y}}(\omega,\omega')+R(d_{T_{X\ast Y}}(\omega,\omega')+2)\\
	&< 4+2N+R(2N+2).
\end{align*}
This result contradicts our initial condition that $d_{X\ast Y}((\omega,z,t),(\omega',z',t'))\ge R(2N+2)+2N+4$. Therefore, our assumption that every term in product \eqref{case1} has absolute value greater than or equal to $\hat{\epsilon}$ must be false, and hence at least one term must have absolute value less than $\hat{\epsilon}$.

Figure \ref{figure15} provides a visual depiction of the scenario in this case.
\begin{figure}[htbp]
	\centering
	\includegraphics[width=0.8\textwidth]{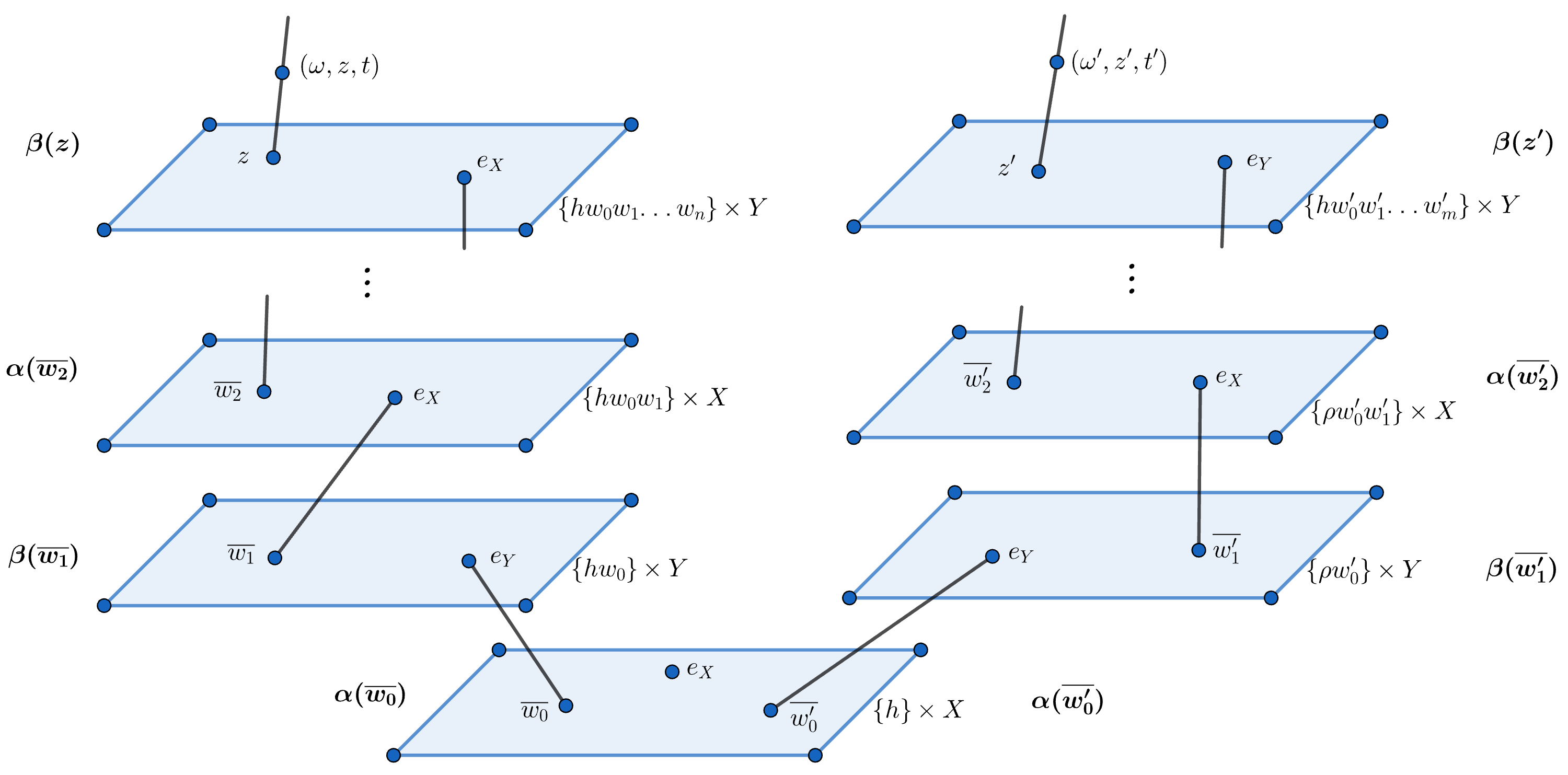}
	\caption{Illustration for Case 1: Points $(\omega, z, t)$ and $(\omega', z', t')$ in $X \ast Y$ with a common prefix $h \neq \epsilon$. $\omega=hw_{0}w_{1}w_{2}\dots w_{n}$ and $\omega'=hw'_{0}w'_{1}w'_{2}\dots w'_{m}$, where $w_{0},w_{0}'\in X_{0}^*$ and $w_{n},w'_{m} \in X_{0}^*$}
	\label{figure15}
\end{figure}

 \textbf{Case 2:} Consider $(\omega,z,t),(\omega',z',t')\in X\ast Y$ with $t,t'>0$. Let $\omega=w_{1}w_{2}\dots w_{n}$ and $\omega'=w'_{1}w'_{2}\dots w'_{m}$ be normal words, assuming $w_{1},w_{n}\in X_{0}^*$ and $w'_{1},w'_{m}\in Y_{0}^*$. Then, the inner product is given by:
\begin{align} \label{eq:case2}
	\langle \gamma(\omega,z,t),\gamma(\omega',z',t') \rangle=&
	\left(\prod_{i=1}^{n}\langle \alpha\beta(\overline{w_{i}}),\omega\rangle\right)  \left(\prod_{j=1}^{m}\langle\omega,\alpha\beta(\overline{w'_{j}}) \rangle\right) \langle\alpha\beta(z),\omega \rangle \langle\alpha\beta(z'),\omega \rangle.
\end{align}

Applying a similar contradiction argument as in Case 1, Suppose every term in product \eqref{eq:case2} has an absolute value greater than or equal to $\hat{\epsilon}$.  By the coarse embeddability of $X$ and $Y$, we infer that $d_{X,Y}(\overline{w_{i}},e_{X,Y})<R$ for $1\leq i\leq n$, $d_{X,Y}(e_{X,Y},\overline{w'_{j}})<R$ for $1\leq j\leq m$, $d_{X,Y}(e_{X,Y},z)<R$, and $d_{X,Y}(e_{X,Y},z')<R$. This leads to the same distance contradiction:
\begin{align*}
	d_{X\ast Y}((\omega,z,t),(\omega',z',t'))
	&\leq 4+d_{T_{X\ast Y}}(\omega,\omega')+R(d_{T_{X\ast Y}}(\omega,\omega')+2)\\
	&< 4+2N+R(2N+2),
\end{align*}
which contradicts $d_{X\ast Y}((\omega,z,t),(\omega',z',t'))\ge R(2N+2)+2N+4$(Figure \ref{figure16} illustrates this case).
\begin{figure}[htbp]
	\centering
	\includegraphics[width=0.8\textwidth]{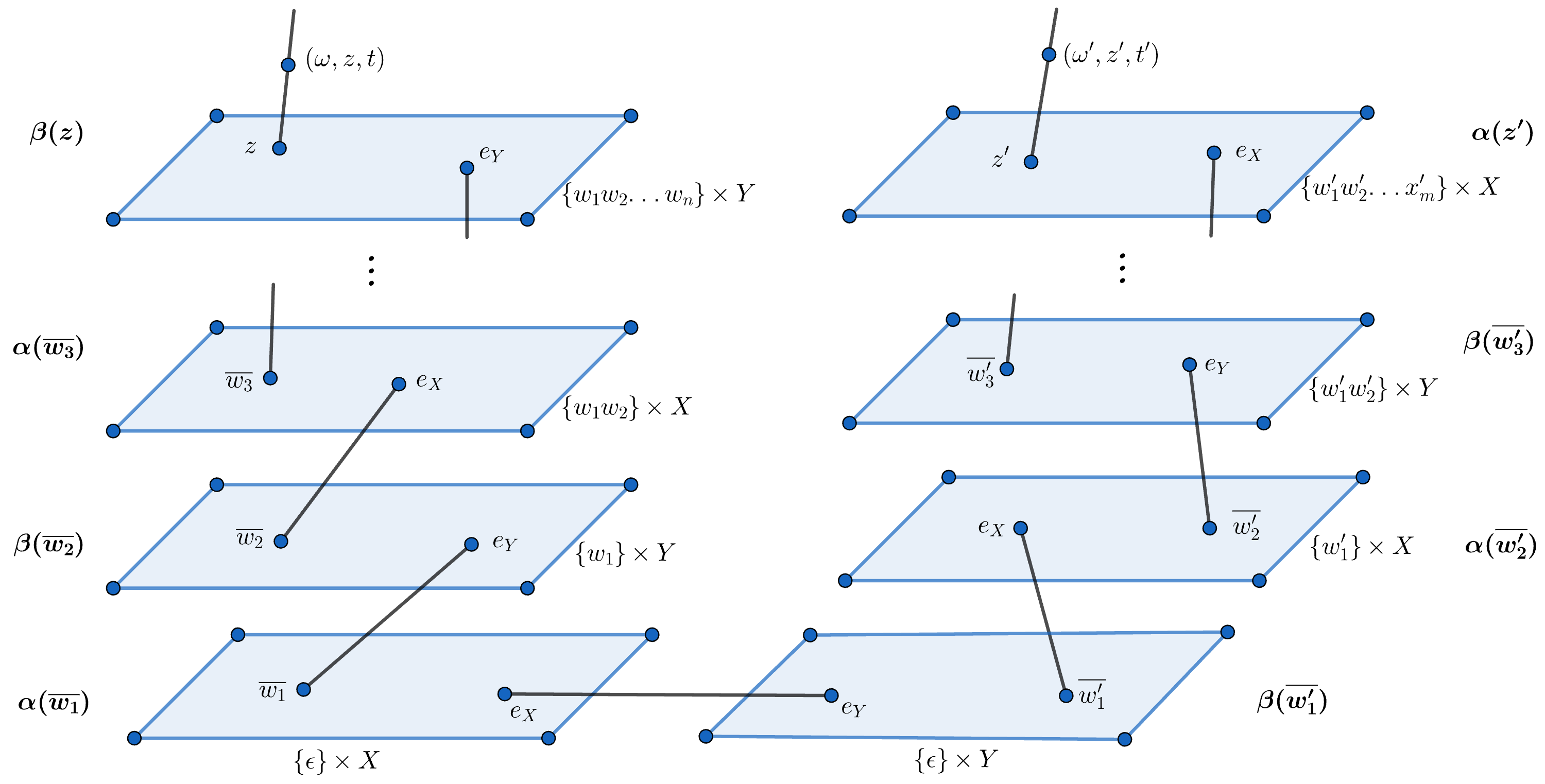}
	\caption{Illustration for Case 2: Points $(\omega, z, t)$ and $(\omega', z', t')$ in $X \ast Y$ with normal word representations $\omega = w_1 w_2 \dots w_n$ and $\omega' = w'_1 w'_2 \dots w'_m$ having no common prefix.}
	\label{figure16}
\end{figure}

 \textbf{Case 3:} Consider $(\omega,z,t),(\omega',z',t')\in X\ast Y$ with $t,t'>0$. Let $\omega=w_{1}w_{2}\dots w_{n}$ and $\omega'=w_{1}w_{2}\dots w_{m}$, assuming $w_{1},w_{n},w_{m}\in X_{0}^*$.
\begin{figure}[htbp]
	\centering
	\includegraphics[width=0.7\textwidth]{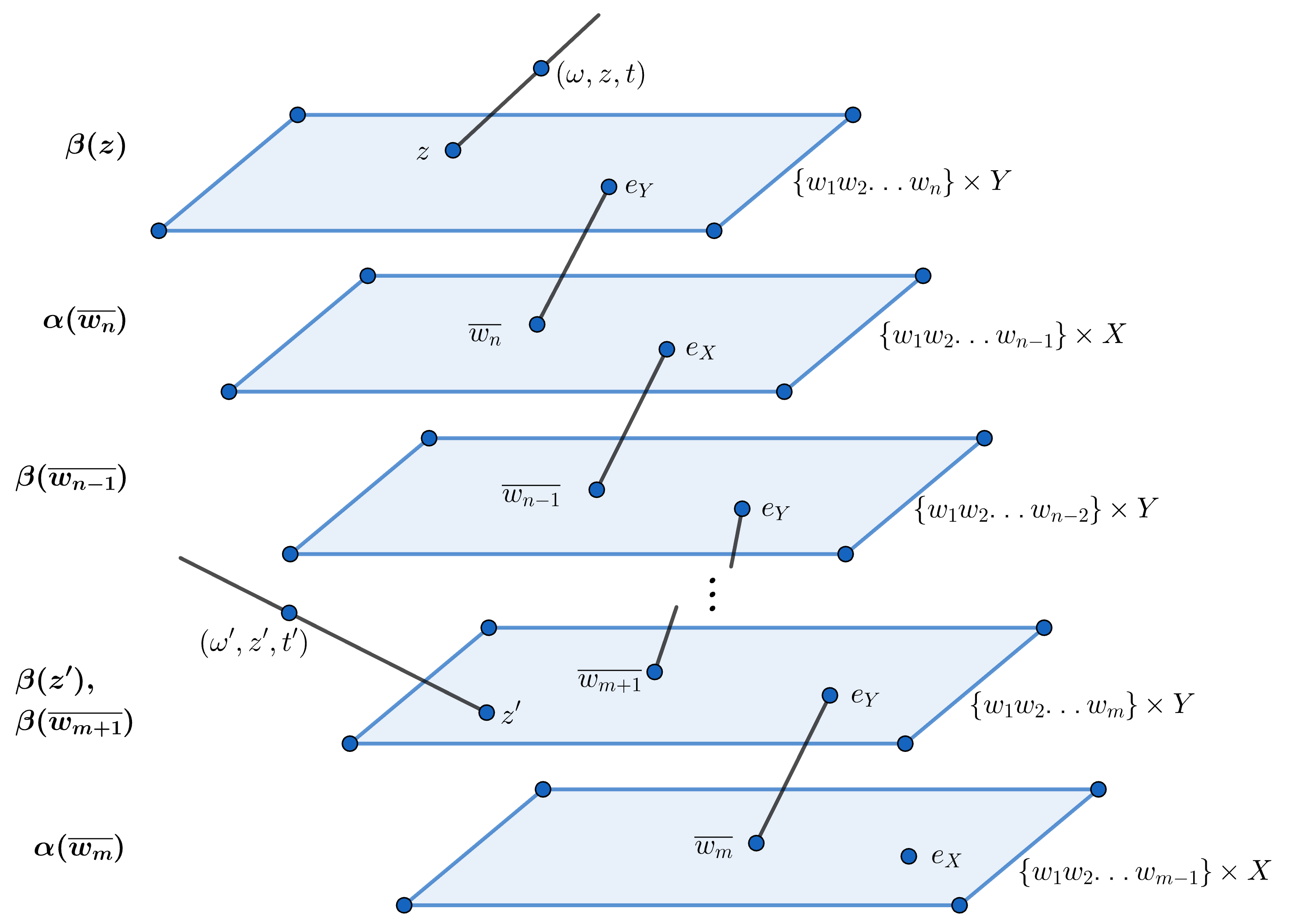}
	\caption{Illustration for Case 3/: Where both $\omega = w_1 w_2 \dots w_n$ and $\omega' = w_1 w_2 \dots w_m$ are normal words sharing a common prefix $w_1 w_2 \dots w_m$ (where $m < n$). The figure visualizes a situation when $n-2 > m$.}
	\label{figure17}
\end{figure}
As the proof method is analogous to the preceding cases, we omit the detailed argument here.

 $\bullet$ \textbf{Convergence Property}

Let $(\omega,z,t),(\omega',z',t')\in X\ast Y$ be points such that their distance $d_{X\ast Y}((\omega,z,t),(\omega',z',t'))$ is bounded by a constant $C$.  We begin by observing a crucial property of the difference between the embeddings $\xi(\omega,z,t)$ and $\xi(\omega',z',t')$. By the definition of the embedding $\xi$ (see \eqref{eq:xi_definition})
$$ \Vert \xi(\omega,z,t)-\xi(\omega',z',t') \Vert^2 = \Vert \gamma(\omega,z,t)-\gamma(\omega',z',t') \Vert^2 +\Vert \tau(\omega)-\tau(\omega') \Vert^2. $$
Taking the square root of both sides, and noting that for non-negative $a,b$, $\sqrt{a^2+b^2} \le a+b$, we get
$$ \Vert \xi(\omega,z,t)-\xi(\omega',z',t') \Vert \leq \Vert \gamma(\omega,z,t)-\gamma(\omega',z',t') \Vert +\Vert \tau(\omega)-\tau(\omega') \Vert. $$
To establish the convergence property, we need to show that this norm is small when $d_{X\ast Y}((\omega,z,t),(\omega',z',t'))\leq C$. We will estimate each term on the right-hand side separately.

First, we consider the term $\Vert \tau(\omega)-\tau(\omega') \Vert$.  By Lemma \ref{tree}, we have a bound on the norm of the difference of $\tau$ embeddings in terms of the tree distance $d_{T}(\omega,\omega')$:
$$\Vert \tau(\omega)-\tau(\omega')\Vert^2\leq \frac{2}{N}d_{T}(\omega,\omega').$$
Since $d_T(\omega, \omega')$ is coarsely bounded by the metric space distance $d_{X\ast Y}((\omega,z,t),(\omega',z',t'))$ (up to an additive constant), we can use the given bound $d_{X\ast Y}((\omega,z,t),(\omega',z',t'))\leq C$ to estimate $d_T(\omega, \omega')$.  Specifically, we use the fact that $d_{T}(\omega,\omega') \leq d_{X\ast Y}((\omega,z,t),(\omega',z',t'))+2$. Thus,
$$\Vert \tau(\omega)-\tau(\omega')\Vert^2\leq \frac{2}{N} \left(d_{X\ast Y}((\omega,z,t),(\omega',z',t'))+2\right) \leq \frac{2}{N} \left(C+2\right) = \frac{2C+4}{N}.$$
Given that we choose $N\ge 32(C+2)\epsilon^{-2}$, taking the square root, we obtain
$$\Vert \tau(\omega)-\tau(\omega')\Vert \leq \frac{\epsilon}{4}.$$

Now we focus on estimating the term $\Vert \gamma(\omega,z,t)-\gamma(\omega',z',t') \Vert$. We will consider this in cases, similar to the Support Property proof, based on the structure of the normal words $\omega$ and $\omega'$.

\noindent \textbf{Case 1:} Let $(\omega,z,t),(\omega',z',t')\in X\ast Y$ with $t,t'>0$. Assume that the normal words $\omega$ and $\omega'$ share a common non-empty prefix $h$, and can be written as $\omega=hw_{0}w_{1}w_{2}\dots w_{n}$ and $\omega'=hw'_{0}w'_{1}w'_{2}\dots w'_{m}$, where $w_{0},w_{0}'\in X_{0}^*$ and $w_{n},w'_{m}\in X_{0}^*$.  We estimate the norm of the difference of $\gamma$ embeddings as follows:
\begin{align*}
	\Vert \gamma(\omega,z,t)-\gamma(\omega',z',t') \Vert
	&\leq \sum_{i=1}^{n}\Vert \alpha\beta(\overline{w_{i}}) - \omega \Vert + \Vert \alpha\beta(\overline{w_{0}}) - \alpha\beta(\overline{w'_{0}})\Vert \\
	&+ \sum_{j=1}^{m}\Vert \alpha\beta(\overline{w'_{j}}) - \omega \Vert + \Vert \alpha\beta(z) - \omega\Vert + \Vert \alpha\beta(z') - \omega\Vert \\
	&\leq \frac{\epsilon}{4C+16}\left(d_{X\ast Y}((\omega,z,t),(\omega',z',t'))+4\right) \\ 
	&\leq \frac{\epsilon}{4C+16}(C+4) \\
	&=\frac{\epsilon}{4}.
\end{align*}
To elaborate on the inequality, we note that each term of the form $\Vert \alpha\beta(\overline{w_{i}})-\omega \Vert$, $\Vert \alpha\beta(\overline{w_{0}})-\alpha\beta(\overline{w'_{0}})\Vert$, $\Vert \alpha\beta(\overline{w'_{j}})-\omega \Vert$, $\Vert \alpha\beta(z)-\omega\Vert$, or $\Vert \alpha\beta(z')-\omega\Vert$ is bounded by $\frac{\epsilon}{4C+16}$. This bound is a direct consequence of the convergence property of the coarse embeddings $\alpha$ and $\beta$. Furthermore, the number of terms appearing in the summation is bounded by $d_{X\ast Y}((\omega,z,t),(\omega',z',t'))+4$.

\noindent \textbf{Case 2:} Let $(\omega,z,t),(\omega',z',t')\in X\ast Y$, $t,t'>0$. Assume $\omega=w_{1}w_{2}\dots w_{n}$ and $\omega'=w'_{1}w'_{2}\dots w'_{m}$, where $w_{1},w_{n}\in X_{0}^*$ and $w'_{1},w'_{m}\in Y_{0}^*$. By using a similar approach, we have: we have:
\begin{align*}
	\Vert \gamma(\omega,z,t)-\gamma(\omega',z',t') \Vert
	&\leq \sum_{i=1}^{n}\Vert \alpha\beta(\overline{w_{i}})-\omega \Vert + \sum_{j=1}^{m}\Vert \alpha\beta(\overline{w'_{j}})-\omega \Vert \\
	&+ \Vert \alpha\beta(z)-\omega\Vert + \Vert \alpha\beta(z')-\omega\Vert \\
	&\leq \frac{\epsilon}{4C+16}\left(d_{X\ast Y}((\omega,z,t),(\omega',z',t'))+4\right) \\ 
	&\leq \frac{\epsilon}{4C+16}(C+4) \\
	&=\frac{\epsilon}{4}.
\end{align*}

\noindent \textbf{Case 3:} Let $(\omega,z,t),(\omega',z',t')\in X\ast Y$, $t,t'>0$. Assume $\omega=w_{1}w_{2}\dots w_{n}$ and $\omega'=w_{1}w_{2}\dots w_{m}$ with $n>m$, and $w_{1},w_{n},w_{m}\in X_{0}^*$. Following the same approach:
\begin{align*}
	\Vert \gamma(\omega,z,t)-\gamma(\omega',z',t') \Vert
	&\leq \Vert \alpha\beta(z)-\omega\Vert + \sum_{i=m+2}^{n}\Vert \alpha\beta(\overline{w_{i}})-\omega\Vert + \Vert\alpha\beta(\overline{w_{m+1}})-\alpha\beta(z')\Vert \\
	&\leq \frac{\epsilon}{4C+16}\left(d_{X\ast Y}((\omega,z,t),(\omega',z',t'))+4\right) \\ 
	&\leq \frac{\epsilon}{4C+16}(C+4) \\
	&=\frac{\epsilon}{4}.
\end{align*}
If $n=m+1$, the middle term vanishes, and the inequality still holds.

In all Cases 1, 2, and 3, we have considered $t, t' > 0$. When $t=0$ or $t'=0$ or both are zero (i.e., points are on sheets), a similar line of reasoning can be applied, using analogous decompositions and bounds based on the coarse embedding properties. The essential idea is to break down the difference $\Vert \gamma(\omega,z,t)-\gamma(\omega',z',t') \Vert$ into a sum of terms that can be controlled by the coarse embedding conditions, ensuring that for points within a bounded distance $C$, this difference remains small, specifically less than $\epsilon/4$.

Finally, combining the estimates for $\Vert \tau(\omega)-\tau(\omega')\Vert$ and $\Vert \gamma(\omega,z,t)-\gamma(\omega',z',t') \Vert$, we have
$$ \Vert \xi(\omega,z,t)-\xi(\omega',z',t') \Vert \leq \Vert \gamma(\omega,z,t)-\gamma(\omega',z',t') \Vert +\Vert \tau(\omega)-\tau(\omega') \Vert \leq \frac{\epsilon}{4} + \frac{\epsilon}{4} = \frac{\epsilon}{2} < \epsilon. $$
This demonstrates the Convergence Property: points in $X\ast Y$ that are close in distance are mapped to points in Hilbert space that are also close, as required for a coarse embedding. 

\end{proof}

 We will now prove that the free product of metric spaces with Property A also has Property A. 
 
\begin{theorem}[Property A for Free Products]\label{theorem_property_a_free_products}
	Let $(X,d_{X})$ and $(Y,d_{Y})$ be metric spaces with nets $(X_{0},i_{X},e_{X})$ and $(Y_{0},i_{Y},e_{Y})$ of $X$ and $Y$, respectively, and let $X\ast Y$ be their free product. If both $X$ and $Y$ have Property A, then $X\ast Y$ also has Property A.
\end{theorem}

 \begin{proof}
 
The proof of this theorem closely parallels that of the Embeddability Theorem, differing only in the verification of the support condition. Instead of showing product \ref{case1} is less than $\hat{\epsilon}$, we now aim to show it equals zero (see \ref{property A}). We will only verify Case 1, as the remaining cases are analogous.

 Consider $(\omega,z,t),(\omega',z',t')\in X\ast Y$, $t,t'>0$, write $\omega=hw_{0}w_{1}w_{2}\dots w_{n}$ and $\omega'=hw'_{0}w'_{1}w'_{2}\dots w'_{m}$, $h\neq \epsilon$, assume $w_{0},w_{0}'\in X_{0}^*$ and $w_{n},w'_{m}\in X_{0}^*$, then we have product \eqref{case1}.

 Since each term in \eqref{case1} has an absolute value not exceeding 1, it is sufficient to prove that at least one term equals 0.  Similarly, we proceed using proof by contradiction. If this were not the case, then every term in the product would have an absolute value greater than 0.  Based on the conditions and the established fact that metric spaces $X$ and $Y$ have Property A, we can deduce that $d_{X,Y}(\overline{w_{i}},e_{X,Y})<R$ for $1\leq i\leq n$, $d_{X,Y}(\overline{w_{0}},\overline{w'_{0}})<R$, $d_{X,Y}(e_{X,Y},\overline{w'_{j}})<R$ for $1\leq j\leq m$, $d_{X,Y}(e_{X,Y},z)<R$, and $d_{X,Y}(e_{X,Y},z')<R$. Consequently, we obtain
 $$d_{X\ast Y}((\omega,z,t),(\omega',z',t'))\leq 4+d_{T_{X\ast Y}}(\omega,\omega')+R(d_{T_{X\ast Y}}(\omega,\omega')+2).$$
 
 This inequality holds because the number of edges in the path from $(\omega,z,t)$ to $(\omega',z',t')$ is at most $4+d_{T_{X\ast Y}}(\omega,\omega')$, and the distance within each sheet is at most $R(d_{T_{X\ast Y}}(\omega,\omega')+2)$. Given our assumption that $d_{T_{X\ast Y}}(\omega,\omega') < 2N$, we derive
 \begin{align*}
 	d_{X\ast Y}((\omega,z,t),(\omega',z',t'))
 	&\leq 4+d_{T_{X\ast Y}}(\omega,\omega')+R(d_{T_{X\ast Y}}(\omega,\omega')+2)\\
 	&< 4+2N+R(2N+2).
 \end{align*}
 
 This result contradicts the initial condition that $d_{X\ast Y}((\omega,z,t),(\omega',z',t')) \ge R(2N+2) + 2N + 4$.
 
\end{proof}

\section{Free products of hyperbolic spaces}

 In this section, we first give a concise proof that the free product of geodesic hyperbolic spaces is a geodesic hyperbolic space. We then show this result holds for general hyperbolic spaces as well.

Let $X$ be a geodesic space. For three points $x, y, z \in X$, the union of geodesics $[x, y] \cup [y, z] \cup [z, x]$ is called a geodesic triangle with vertices $\{x, y, z\}$, denoted by $\triangle(x, y, z)$. The geodesics $[x, y]$, $[y, z]$, and $[z, x]$ are called the sides of the triangle.

\begin{definition}
	Let $\delta \geq 0$. For any three points $x, y, z \in X$, if each side of the geodesic triangle with these points as vertices is contained within the $\delta$-neighborhood of the union of the other two sides, that is,
	\begin{align*}
		[x, y] &\subset N([y, z] \cup [z, x]; \delta), \\
		[y, z] &\subset N([z, x] \cup [x, y]; \delta), \\
		[z, x] &\subset N([x, y] \cup [y, z]; \delta),
	\end{align*}
	holds, then $X$ satisfies the $\delta$-Rips condition.
\end{definition}

\begin{definition}
	A geodesic space $X$ is a geodesic Gromov hyperbolic space if there exists a constant $\delta$ such that it satisfies the $\delta$-Rips condition.
\end{definition}

\begin{theorem}
	
	If $X$ and $Y$ are geodesic hyperbolic spaces, then their free product $X\ast Y$ is also a geodesic hyperbolic space.

\end{theorem}

\begin{proof}
	By the definition of the free product of metric spaces, the free product of two geodesic metric spaces is readily seen to be geodesic. Consequently, our focus need only be on verifying hyperbolicity. Without loss of generality, we consider two cases, as shown in Figure \ref{geodesic}. Since $X$ and $Y$ are both hyperbolic spaces, by definition, there exist corresponding constants $\delta_X \geq 0$ and $\delta_Y \geq 0$ such that the metric spaces $X$ and $Y$ each satisfy the $\delta$-Rips condition for $\delta_X$ and $\delta_Y$ respectively. We define $\delta$ as $\max\{\delta_X, \delta_Y\}$.

		\begin{figure}[htbp]
		\centering
		\includegraphics[width=0.8\textwidth]{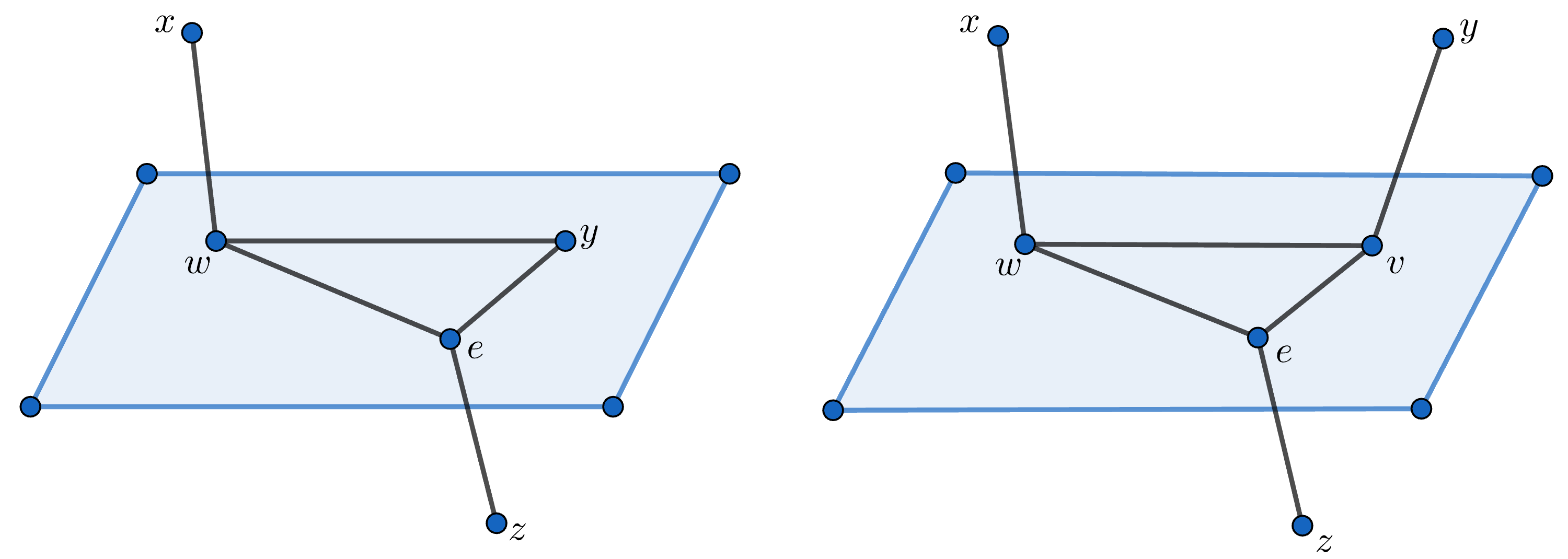}
		\caption{The left figure illustrates Case 1, while the right figure illustrates Case 2.}
		\label{geodesic}
	\end{figure}
	
	\textbf{Case 1:}
	
	
	It suffices to show that the geodesic $[x,y]$ is contained in the $\delta$-neighborhood of the union of the other two geodesics. The geodesic $[x,y]$ coincides with the geodesic $[x,z]$ along the segment $[x,w]$. Therefore, any point in $[x,w]$ trivially lies within the $\delta$-neighborhood of the union of the other two geodesics. Consequently, we need only consider points in $[w,y]$. Since $X$ and $Y$ are hyperbolic spaces, and the segments $[w,e]$ and $[y,e]$ are part of the geodesics $[x,z]$ and $[y,z]$ respectively, it follows that any point in $[w,y]$ is contained within the $\delta$-neighborhood of the union of $[w,e]$ and $[e,y]$. Hence, we obtain the result
	\[ [x, y] \subset N([y, z] \cup [z, x]; \delta) \]
	Using the same argument, we deduce that
	\[ [y, z] \subset N([z, x] \cup [x, y]; \delta) \]
	and
	\[ [z, x] \subset N([x, y] \cup [y, z]; \delta) \]

	\textbf{Case 2:}
	
	
	As in Case 1, it suffices to show that the geodesic $[x,y]$ is contained in the $\delta$-neighborhood of the union of the other two geodesics.
The geodesic $[x,y]$ has a common segment $[x,w]$ with the geodesic $[x,z]$. Thus, for any point $p \in [x,w]$, we have $p \in [x,z] \subset [x,z] \cup [y,z]$, and the condition is trivially satisfied. Similarly, $[x,y]$ shares a common terminal segment $[v,y]$ with $[y,z]$. Hence, any point in $[v,y]$ is also contained within the $\delta$-neighborhood of $[x,z] \cup [y,z]$. Therefore, we need only consider the remaining segment $[w,v]$ of $[x,y]$. By the hyperbolicity of $X$ and $Y$, any point in $[w,v]$ lies within the $\delta$-neighborhood of the union of segments $[w,e]$ and $[v,e]$. Moreover, since $[w,e] \subset [x,z]$ and $[v,e] \subset [y,z]$, we obtain the result
\[ [x, y] \subset N([y, z] \cup [z, x]; \delta). \]
Using the same argument, we deduce that
\[ [y, z] \subset N([z, x] \cup [x, y]; \delta) \]
and
\[ [z, x] \subset N([x, y] \cup [y, z]; \delta). \]

\end{proof}

We now shift our focus to the more general setting where $X$ is merely assumed to be a metric space. The previous assumption that $X$ is a geodesic space is no longer imposed. With this broader context, we proceed to redefine hyperbolicity.

\begin{definition}
	For three points $x, y, z \in X$, we define a non-negative real number $(y \mid z)_x$ as follows:
	$$(y \mid z)_x := \frac{1}{2}(\overline{y,x} + \overline{z,x} - \overline{y,z})$$
	and we call it the Gromov product of $y$ and $z$ with respect to the base point $x$. If the choice of the base point $x$ is clear from the context, we may omit the base point and write $(y \mid z)$. Additionally, we remark that the notation $\overline{x,y}$ is used in place of $d(x,y)$ for notational convenience.
\end{definition}

\begin{definition}
	A metric space $X$ is called $\delta$-hyperbolic if there exists a non-negative real number $\delta \geq 0$ such that for any four points $x, y, z, w \in X$, the following inequality holds:
	
	$$(x \mid z)_w \geq \min\{(x \mid y)_w, (y \mid z)_w\} - \delta.$$
	
	If a metric space $X$ is $\delta$-hyperbolic for some $\delta \geq 0$, it is also called a Gromov hyperbolic space. Furthermore, $X$ may be referred to simply as hyperbolic when the specific value of $\delta$ is not important.
\end{definition}

To establish a complete logical framework, we introduce the following proposition.

\begin{proposition}\cite{camps2008einführung}
	Let $X$ be a geodesic space. The following two conditions on $\delta \geq 0$ are equivalent, with constants linearly related to each other:
	
	\begin{enumerate}
		\item (Hyperbolicity): For any four points $x, y, z, w \in X$, the following inequality holds:
		\[ (x \mid z)_w \geq \min \{ (x \mid y)_w, (y \mid z)_w \} - \delta \]
		where $(y \mid z)_x := \frac{1}{2} (\overline{y, x} + \overline{z, x} - \overline{y, z})$ is the Gromov product.
		
		\item ($\delta$-Rips Condition): Every geodesic triangle in $X$ is $\delta$-slim. This means for any geodesic triangle $\Delta = [x,y] \cup [y,z] \cup [z,x]$ in $X$, and any point $u$ on $[x,y]$, there exists a point $v$ on $[x,z] \cup [y,z]$ such that $\overline{u,v} \leq \delta$.
	\end{enumerate}

	More precisely, the hyperbolicity condition with constant $\delta$ implies the $\delta$-Rips condition with constant $4\delta$, and the $\delta$-Rips condition with constant $\delta$ implies the hyperbolicity condition with constant $2\delta$.

\end{proposition}

We now present a straightforward result, which will be used in later proofs. This is stated as the following lemma.

\begin{lemma}\cite{coornaert1990geometrie}
	\label{lem:hyper}
    If $(X, x_0)$ is $\delta$-hyperbolic, then $(X, w)$ is $2\delta$-hyperbolic for any point $w \in X$.
\end{lemma}

\begin{proof}
	Let $x, y, z \in X$ be arbitrary, and let $w$ be any point in $X$. Our goal is to show that $(x \mid y)_{w} \geq \min \{(x \mid z)_{w}, (y \mid z)_{w} \} - 2\delta$.
	
	Since $(X, x_0)$ is $\delta$-hyperbolic, we have
	$$(x\mid y)_{x_0} \geq \min\{(x\mid z)_{x_0}, (y\mid z)_{x_0}\} - \delta.$$
	
	Among the numbers $(x \mid z)_{x_0}$, $(x \mid w)_{x_0}$, and $(y \mid z)_{x_0}$, we can assume (by possibly exchanging $x$ and $y$, or $z$ and $w$) that $(x \mid z)_{x_0}$ is the largest. Then, by the $\delta$-hyperbolicity condition stated above:
	
	\[ (x\mid y)_{x_0} \geq \min \{ (x \mid z)_{x_0}, (y \mid z)_{x_0} \} - \delta = (y \mid z)_{x_0} - \delta \]
	and
	\[ (z\mid w)_{x_0} \geq \min \{ (z \mid x)_{x_0}, (w \mid x)_{x_0} \} - \delta = (w \mid x)_{x_0} - \delta. \]
	Adding these inequalities, we get:
	$$(x \mid y)_{x_0} + (z \mid w)_{x_0} \geq  (y \mid z)_{x_0} + (w \mid x)_{x_0} - 2\delta$$
	Thus, we have established the following relationship:
	$$(x \mid y)_{x_0} + (z \mid w)_{x_0} \geq \min \{ (x \mid z)_{x_0} + (y \mid w)_{x_0}, (x \mid w)_{x_0} + (y \mid z)_{x_0} \} - 2\delta$$
	We add to both sides of the inequality the quantity:
	\[ (\overline{x, w} + \overline{y, w} + \overline{z, w}- \overline{x, x_0} - \overline{y, x_0}-\overline{z, x_0}-\overline{w, x_0})/2 \]
	which yields $(x \mid y)_{w} \geq \min \{ (x \mid z)_{w}, (y \mid z)_{w} \} - 2\delta$. Thus, $(X, w)$ is $2\delta$-hyperbolic as desired.
	
\end{proof}

\begin{remark}
	In order to verify if a metric space fulfills the Gromov hyperbolicity condition, it suffices to verify for one fixed base point.
\end{remark}

\begin{theorem}
		If $X$ and $Y$ are hyperbolic spaces(not necessarily geodesic), then their free product $X\ast Y$ is also a hyperbolic space.
\end{theorem}

\begin{proof}
      Without loss of generality, we consider two cases, as shown in Figure \ref{general}.  For convenience, we denote the basepoint of a certain sheet of $X \ast Y$ by $o$.
      	\begin{figure}[htbp]
      	\centering
      	\includegraphics[width=0.8\textwidth]{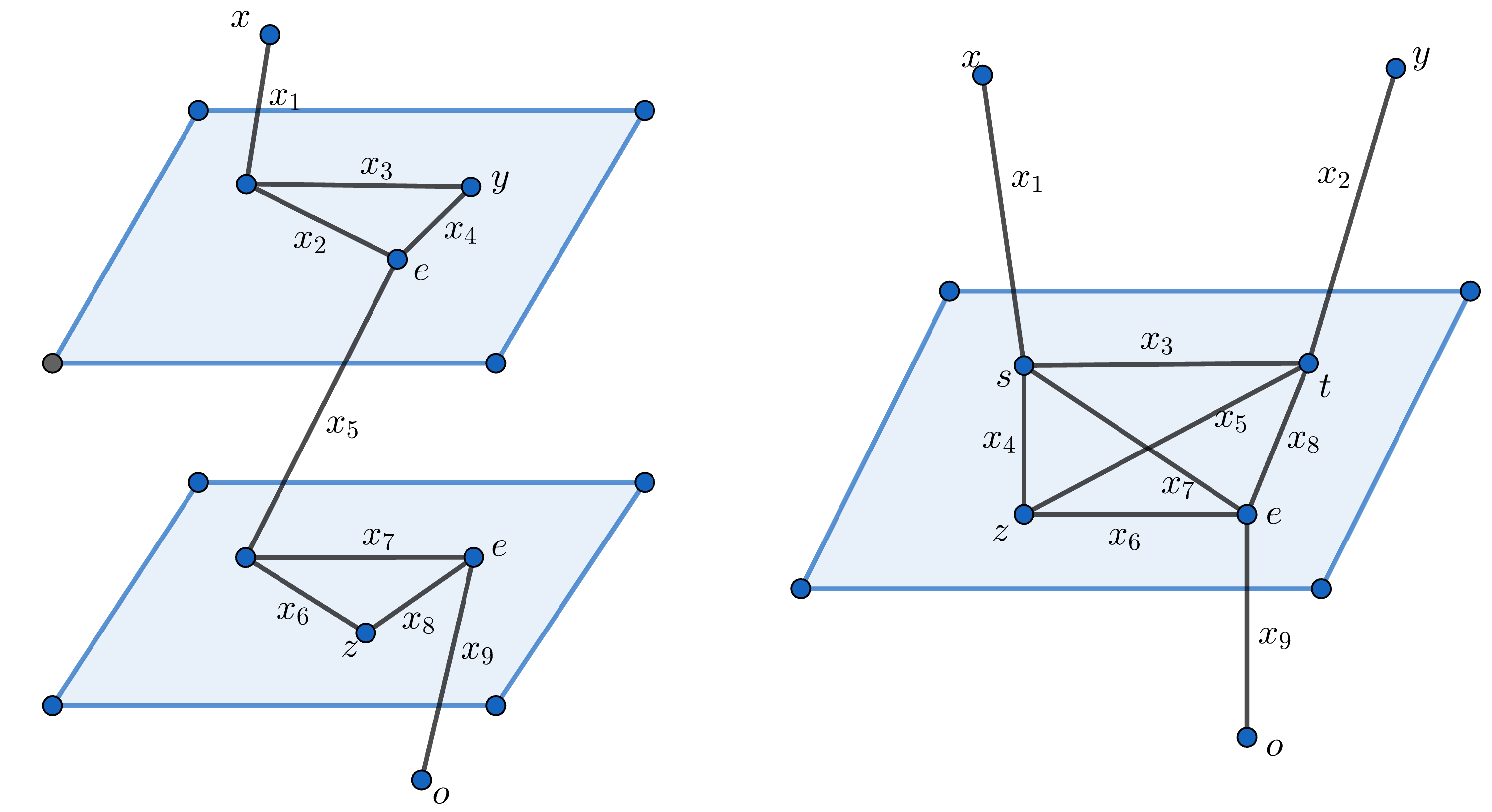}
      	\caption{The left figure illustrates Case 1, while the right figure illustrates Case 2.}
      	\label{general}
      \end{figure}

      \textbf{Case 1:}
      
    	By Lemma~\ref{lem:hyper}, it suffices to prove the Gromov hyperbolicity condition for a single fixed point. We choose this fixed point to be $o$. That is, we only need to verify the existence of a constant $\delta$ such that the following inequalities hold:
     	\begin{align*}
     		(x \mid y)_o &\geq \min \{ (x \mid z)_o, (y \mid z)_o \} - \delta, \\
     		(x \mid z)_o &\geq \min \{ (x \mid y)_o, (y \mid z)_o \} - \delta, \\
     		(y \mid z)_o &\geq \min \{ (x \mid y)_o, (x \mid z)_o \} - \delta.
     	\end{align*}
     	We will proceed to calculate the Gromov products $(x \mid y)_o$, $(x \mid z)_o$, and $(y \mid z)_o$:
     \begin{align*}
     	(x \mid y)_o &= \frac{1}{2} (\overline{x, o} + \overline{y, o} - \overline{x, y}) = \frac{1}{2}(x_{2}+x_{4}-x_{3})+x_{5}+x_{7}+x_{9}; \\
     	(x \mid z)_o &= \frac{1}{2} (\overline{x, o} + \overline{z, o} - \overline{x, z}) = \frac{1}{2}(x_{7}+x_{8}-x_{6})+x_{9}; \\
     	(y \mid z)_o &= \frac{1}{2} (\overline{y, o} + \overline{z, o} - \overline{y, z}) = \frac{1}{2}(x_{7}+x_{8}-x_{6})+x_{9}.
     \end{align*}
     	We observe that $(x \mid z)_o = (y \mid z)_o = \frac{1}{2}(x_{7}+x_{8}-x_{6})+x_{9}$. Furthermore, since
     	\[ (x_{2}+x_{4}-x_{3})+x_{5}+x_{7}+x_{9} \geq (x_{7}+x_{8}-x_{6})+x_{9}, \]
     	we clearly have
     	\[ (x \mid y)_o \geq \min \{ (x \mid z)_o, (y \mid z)_o \}. \]
     	On the other hand,
     	\[ (x \mid z)_o = \min \{ (x \mid y)_o, (y \mid z)_o \} \]
     	because
     	\[ \min \{ (x \mid y)_o, (y \mid z)_o \} = (y \mid z)_o. \]
     	The equality
     	\[ (y \mid z)_o = \min \{ (x \mid y)_o, (x \mid z)_o \} \]
     	follows by the same reasoning. Therefore, in the case of Case 1, the Gromov hyperbolicity condition for a single fixed point is satisfied, and such a constant $\delta$ exists.

 \textbf{Case 2:}     
     
     	As in Case 1, we only need to verify the existence of a constant $\delta$ such that the following inequalities hold:
     	\begin{align*}
     		(x \mid y)_o &\geq \min \{ (x \mid z)_o, (y \mid z)_o \} - \delta, \\
     		(x \mid z)_o &\geq \min \{ (x \mid y)_o, (y \mid z)_o \} - \delta, \\
     		(y \mid z)_o &\geq \min \{ (x \mid y)_o, (x \mid z)_o \} - \delta.
      	\end{align*}
       	We will proceed to calculate the Gromov products $(x \mid y)_o$, $(x \mid z)_o$, and $(y \mid z)_o$:
        \begin{align*}
        	(x \mid y)_o &= \frac{1}{2} (\overline{x, o} + \overline{y, o} - \overline{x, y}) = \frac{1}{2}(y_{7}+y_{8}-y_{3})+y_{9}; \\
        	(x \mid z)_o &= \frac{1}{2} (\overline{x, o} + \overline{z, o} - \overline{x, z}) = \frac{1}{2}(y_{7}+y_{6}-y_{4})+x_{9}; \\
        	(y \mid z)_o &= \frac{1}{2} (\overline{y, o} + \overline{z, o} - \overline{y, z}) = \frac{1}{2}(y_{8}+y_{6}-y_{5})+x_{9}.
      \end{align*}

      By observation, we find that $\frac{1}{2}(y_{7}+y_{8}-y_{3})=(s \mid t)_e$, $\frac{1}{2}(y_{7}+y_{6}-y_{4})=(s \mid z)_e$, and $\frac{1}{2}(y_{8}+y_{6}-y_{5})=(sz\mid t)_e$. Since $s, t, z, e$ are in the same hyperbolic space, by the definition of a hyperbolic space, there exists a constant $\delta$ such that
       \begin{align*}
      	   (s\mid z)_e &\geq \min \{ (s \mid t)_e, (t \mid z)_e \} - \delta, \\
           (s \mid t)_e &\geq \min \{ (s \mid z)_e, (t \mid z)_e \} - \delta, \\
           (t \mid z)_e &\geq \min \{ (t \mid s)_e, (s \mid z)_e \} - \delta.
      \end{align*}
      Thus, we have completed the proof.

\end{proof}

\section*{Acknowledgement}
The authors are very grateful to Professor Tomohiro Fukaya for his careful reading of the article, which led to the correction of an error in the first version. Furthermore, his thoughtful feedback and an interesting question were very inspiring for further research on this subject. The authors are also grateful to Professor Jiawen Zhang and Doctor Liang Guo for many helpful suggestions.

\bibliographystyle{plain}
\bibliography{freeproduct.bib}

\bigskip

\address{Qin Wang  \endgraf
	School of Mathematical Science,
	East China Normal University,
	Shanghai, 200241, China.
}

\textit{E-mail address}: \texttt{qwang@math.ecnu.edu.cn}

\bigskip

\address{Jvbin Yao \endgraf
School of Mathematical Science,
East China Normal University,
Shanghai, 200241, China.
}

\textit{E-mail address}: \texttt{52285500009@stu.ecnu.edu.cn}                 




\end{document}